\pdfoutput=1

    \documentclass[11pt,twoside]{article}

    \usepackage[kerning,spacing]{microtype}
    \microtypecontext{spacing=nonfrench}

    \usepackage{graphicx}
    \usepackage{styleset}
    \usepackage{macros}
    \let\subsubsection\subparagraph

    \title  {Filtrations on instanton homology}

    \author {P. B. Kronheimer and T. S. Mrowka%
                
            \thanks{%
            The work of the first author
            was supported by the National Science Foundation through
            NSF grant DMS-0904589. The work of
            the second author was supported by NSF grant DMS-0805841.}} 

    \address {Harvard University, Cambridge MA 02138 \\
              Massachusetts Institute of Technology, Cambridge MA 02139}       

\begin{document}

\maketitle
 
\tableofcontents

\section{Introduction}

The Khovanov cohomology $\kh(K)$ of an oriented knot or link is defined
in \cite{Khovanov} as the cohomology of a cochain complex $(\bC=\bC(D),
d_{\kh})$ associated to a plane diagram $D$ for $K$:
\[
        \kh(K) = H\bigl(\bC, d_{\kh}\bigr).
\]
The free abelian group $\bC$ carries both a cohomological grading
$h$ and a quantum grading $q$. The differential $d_{\kh}$ increases
$h$ by $1$ and preserves $q$, so that the Khovanov cohomology is
bigraded. We write $\cF^{i,j}\bC$ for the decreasing filtration
defined by the bigrading, so $\cF^{i,j}\bC$ is generated by
elements whose cohomological grading is not less than $i$ and whose
quantum grading is not less than $j$. In general, given abelian groups
with a decreasing filtration indexed by $\Z\oplus\Z$, we
will say that a group homomorphism $\phi$ has \emph{order} 
$\ge(s,t)$ if $\phi(\cF^{i,j})\subset \cF^{i+s,j+t}$. So
$d_{\kh}$ has order $\ge(1,0)$.

In \cite{KM-unknot}, a new invariant $\Isharp(K)$ was defined using
singular instantons, and it was shown that $\Isharp(K)$is related to
$\kh(K^{\dag})$ through a spectral sequence. The notation
$K^{\dag}$ here denotes the mirror image of $K$. Building on the
results of \cite{KM-unknot}, we establish the following theorem in
this paper.

\begin{theorem}\label{thm:main}
    Given an oriented  link $K$ in $\R^{3}$, and a
    diagram $D^{\dag}$ for $K^{\dag}$,
    one can
    construct  a differential $d_{\sharp}$ on the Khovanov complex
    $\bC=\bC(D^{\dag})$ such that the homology of $(\bC, d_{\sharp})$ is
    $\Isharp(K)$. The differential $d_{\sharp}$ is equal to
    $d_{\kh}=d_{\kh}(D^{\dag})$ 
    to leading order in the $q$-filtration: 
    that is both differentials have order $\ge(1,0)$, 
    and the difference $d_{\sharp} - d_{\kh}$ has
    order $\ge(1,2)$.
\end{theorem}

The differential $d_{\sharp}$ depends (a priori) on more than just a
choice of diagram for $K$. It depends also on choices of
perturbations and metrics, required to make moduli spaces of
instantons transverse. The fact that, with appropriate choices, 
the complex which computes
$\Isharp$ can be made to coincide with $\bC$ as an abelian group is
easily seen from the authors' earlier paper  \cite{KM-unknot}. 
The new content in the above theorem is that
$d_{\sharp}-d_{\kh}$ has order $\ge2$ with respect to the quantum
filtration. (The quantum degree is of constant parity on the complex
$\bC$, so order $\ge2$ is inevitable once the leading-order parts agree.)

 As a consequence of the above theorem, 
the filtrations by $i$ and $j$ lead to two
    spectral sequences abutting to $\Isharp(K)$, whose $E_{2}$
    and $E_{1}$ terms respectively are both isomorphic to
    $\kh(K^{\dag})$. Indeed, there is such a spectral sequence for every
    positive linear combination of $i$ and $j$. The next theorem
    addresses the topological invariance of these spectral sequences.
    We write $\mathfrak{C}$ for the category whose objects are
    finitely-generated differential abelian groups filtered by
    $\Z\oplus\Z$ with differentials of order $\ge(1,0)$, and whose
    morphisms are differential group homomorphisms of order $(0,0)$ up
    to chain-homotopies of order $\ge(-1,0)$.

\begin{theorem}\label{thm:invariance}
   In the category
    $\mathfrak{C}$, the isomorphism class of $(\bC, d_{\sharp})$
    depends only on the oriented link $K$.
\end{theorem}

From the above theorem, it follows that the various pages of the
resulting spectral sequences are invariants of $K$, as the next
corollary states. (A similar result for a related spectral sequence
\cite{OS-double-covers} involving the Heegaard Floer homology of the branched
double cover was established by Baldwin \cite{Baldwin}.)

\begin{corollary}
    For the homological and quantum  filtrations by $i$
    and $j$, the isomorphism type of the all the pages $(E_{r},
    d_{r})$, for $r\ge 2$ or $r\ge 1$ respectively, are invariants of
    $K$. More generally, for the filtration by $a i + bj$ with
    $a,b\ge 1$, the same is true when $r \ge a+1$.
\end{corollary}

\begin{proof}[Proof of the corollary]
    See  for example \cite[Chapter IX, Proposition 3.5]{Maclane}, though
    the notation there uses increasing filtrations rather than the
    decreasing filtrations of this paper.
\end{proof}

\begin{corollary}
    The homological and quantum filtrations $i$ and $j$ give rise to
    filtrations of the instanton homology $\Isharp(K)$, as do the
    combinations $ai + bj$. These filtrations are invariants of $K$. \qed
\end{corollary}

Our results leave open a functoriality question for the pages
$(E_{r},d_{r})$ of the spectral sequences. For example,
Theorem~\ref{thm:invariance} does not imply that there is a functor to
$\mathfrak{C}$ from the category whose objects are oriented links and
whose morphisms are isotopies.  We do expect that a result of this
sort is true however. The issue is similar to the ones that arise in
proving that Khovanov cohomology is functorial \cite{Jacobsson}.
  
We do know that the homology groups $\Isharp(K)$ are functorial for
cobordisms \cite{KM-unknot}.  Thus, if $K_{1}$ and $K_{0}$ are
oriented links and $S \subset [a,b]\times \R^{3}$ is a cobordism from
$K_{1}$ to $K_{0}$ (which we allow to be non-orientable), then there
is a map
\[
            \Isharp(S) :\Isharp(K_{1}) \to
           \Isharp(K_{0})
\]
which is well-defined up to an overall sign. The following proposition
describes how the filtrations behave under such a map.

    \begin{proposition}\label{prop:order-of-map}
        The map $\Isharp(S) :\Isharp(K_{1}) \to
           \Isharp(K_{0})$ resulting from a cobordism $S$ is represented
        at the chain level by a map $\bC_{1} \to \bC_{0}$ of order
    \[ \ge \left(\tfrac{1}{2}(S\cdot S), \chi(S)+\tfrac{3}{2}(S\cdot S)\right) . \]
Here the term $S\cdot S$ is the self-intersection number of the
surface $S$ defined with respect to a push-off which, at the
two ends, has total linking number $0$ with $K_{1}$ and $K_{0}$
respectively. 
    \end{proposition}

Note that the self-intersection number which appears in the
proposition 
is zero if $S$ is an
oriented cobordism, and is always even. These results for the filtrations of
$\Isharp$ should be compared to the corresponding statements for
Khovanov homology, where it is known that an orientable cobordism $S$
gives rise to a map that preserves the homological grading and maps
elements of quantum degree $j$ to elements of degree $j + \chi(S)$.

\begin{remark}
    There is also a reduced version of the instanton homology, denoted
    by $\Inat(K)$ in \cite{KM-unknot}, which is related to the reduced
    Khovanov homology $\khr(K^{\dag})$.  Theorem~\ref{thm:main} and
    Proposition~\ref{prop:order-of-map} can be formulated for these
    reduced versions with no essential change to the wording.
\end{remark}

The remainder of this paper is organized as follows. In sections
\ref{subsec:unlinks} through \ref{sec:reidemeister}, we focus on the
$q$-filtration. Section~\ref{subsec:unlinks} introduces a quantum
filtration on the instanton homology of an
unlink. 
Section~\ref{sec:cube} reviews the ``cube of resolutions''
in the context of instanton homology, from \cite{KM-unknot}, and this
is used in the following section to extend the $q$-filtration to the
case of a general link. Rather than working with traditional diagrams
(plane projections) of links, we introduce the slightly more flexible
notion of  a \emph{pseudo-diagram}
(Definition~\ref{def:pseudo-diagram}). With one additional hypothesis
on the pseudo-diagram, we prove that the differential on the cube
complex preserves the $q$-filtration
(Proposition~\ref{prop:C-is-filtered}). Sections \ref{sec:isotopy} through
\ref{sec:dropping-crossing} examine how the $q$-filtration behaves when a
pseudo-diagram is altered by an isotopy or by adding or dropping
crossings, and we can then treat Reidemeister moves by regarding a
single Reidemeister move as a sequence of isotopies and add-drops of
crossings.  The $h$-filtration is somewhat simpler to deal with than
the $q$-filtration, but follows the same outline: it is discussed in
section~\ref{sec:h}. The proofs of the main results are then given in
section~\ref{sec:proofs}. The final section contains some simple
examples, to illustrate the use of pseudo-diagrams, as well as a more
complicated example: the $(4,5)$-torus knot.

\section{Unlinks}
\label{subsec:unlinks}

Let $K \subset \R^{3}$ be an oriented link that is isotopic to the
standard $p$-component unlink $U_{p}$. According to
\cite[Proposition 8.11]{KM-unknot}, we have an isomorphism
\begin{equation} \label{eq:unknot-V-tensor-p}
\Isharp(K) \stackrel{\gamma}{\to} V^{\otimes p} 
\end{equation}
that depends only on the given orientation 
and a choice of ordering of the
components of $K$. Here
\[
        V = \langle \vp, \vm \rangle
\]
is a free abelian group on two generators. This isomorphism arises in
\cite{KM-unknot} as a consequence of an excision property which is
used to establish a K\"unneth product formula for $\Isharp$ of a split
link $K_{1}\amalg K_{2}$. 

On the other hand, there is a more direct
way to compute $\Isharp(K)$ for an unlink $K$, working from the
definition of instanton homology. We recall the definition
in outline. From $K$, we form a new
link $K^{\sharp}\subset S^{3}$, the union of $K$ and a standard Hopf link
near infinity, and we equip $S^{3}$ with an orbifold metric with
cone-angle $\pi$ along $K^{\sharp}$. We let $\omega$ denote an arc joining the two
components of the Hopf link, and we form the configuration space
$\conf^{\omega}(S^{3},K^{\sharp})$ consisting of singular $\SO(3)$ connections
on the complement of $K^{\sharp}$, with $w_{2}$ equal to the dual of
$\omega$ and with holonomy asymptotically conjugate to the element
\[
                    \bi =
                    \begin{pmatrix}
                        i & 0 \\ 0 & -i 
                    \end{pmatrix}
\]
on small circles linking $K^{\sharp}$. We form $\bonf^{\omega}(S^{3},
K^{\sharp})$ as the quotient of $\conf^{\omega}(S^{3},K^{\sharp})$ by
the determinant-1 gauge group. In
$\bonf^{\omega}(S^{3},K^{\sharp})$ we consider the critical points of the perturbed
Chern-Simons functional $\CS + f$, where $f$ is a holonomy
perturbation chosen to achieve a
Morse-Smale transversality condition for the formal gradient flow.
For the \emph{unperturbed} Chern-Simons functional, the set of
critical points is the space of flat connections in
$\bonf^{\omega}(S^{3}, K^{\sharp})$, and these can be identified with
the representation variety
\[
           \Rep(K , \bi)
\]
consisting of homomorphisms $\rho: \pi_{1}(\R^{3} \sminus K) \to
\SU(2)$ with $\rho(m) \sim \bi$, for all meridians $m$ of $K$.

In the case of an unlink, the fundamental group of $\R^{3}\sminus K$
is free on $p$ generators: we can specify generators by  giving
explicit choices of meridians, oriented consistently with the given
orientation of $K$. After making these choices, we have an
identification
\[
       \Rep(K,\bi) = (S^{2})^{p},
\]
(where the $2$-sphere is the conjugacy class of $\bi$ in
$\SU(2)$). This representation variety sits in $\bonf^{\omega}(S^{3},
K^{\sharp})$ as a Morse-Bott critical set for $\CS$. 
A product of $2$-spheres carries an obvious Morse function
with critical points only in even index. By choosing $f$ above so that
its restriction to $\Rep(K,\bi)$ is equal to such a Morse function, we
can arrange that the critical points of $\CS + f$ consist of exactly
$2^{p}$ points, all in the same index mod $2$. The differential in the
instanton homology is then zero, and $\Isharp(K)$ is the free abelian
group generated by these critical points. Thus we obtain an
isomorphism,
\begin{equation}
    \label{eq:rep-homology}
     \Isharp(K) \stackrel{\beta}{\to} H_{*}(S^{2})^{\otimes p}
\end{equation}
as the composite
\begin{equation*}
    \begin{aligned}
        \Isharp(K) &= H_{*}(\Rep(K,\bi)) \\
        &= H_{*}((S^{2})^{p}) \\
        &= H_{*}(S^{2})^{\otimes p}.
    \end{aligned}
\end{equation*}

We now have two different ways to identify $\Isharp(K)$ with the
tensor product of $p$ copies of a free abelian group of rank $2$,
through the isomorphisms  $\gamma$ and $\beta$ of equations 
\eqref{eq:unknot-V-tensor-p} and
\eqref{eq:rep-homology} respectively. Combining the first isomorphism with inverse
of the second, we have a map
\begin{equation}\label{eq:Morse-to-V}
         \epsilon = \gamma\comp\beta^{-1}: 
       H_{*}(S^{2})^{\otimes p} \to V^{\otimes p}.
\end{equation}
Using the $\Z/4$ grading on instanton homology (for example) it is easy to see that, in the
case $p=1$, this map is the isomorphism
\[
H_{*}(S^{2}) \to V
\]
that sends the $2$-dimensional generator to $\vp$ and the
$0$-dimensional generator to $\vm$ (given appropriate conventions
about orientations, to fix the signs). For larger $p$, it does not
follow that this map is simply the $p$'th tensor power of
the isomorphism from the $p=1$ case. (The map potentially involves
instantons on the cobordisms that are used in the proof of the
excision property.) Indeed, $\epsilon$ will in general depend on the
choice of metric and perturbation.

What we \emph{can} say about $\epsilon$ is this.
    Make $V$ a graded abelian group by putting $\vp$ and $\vm$ in
    degrees $1$ and $-1$ respectively, and give $V^{\otimes p}$ the
    tensor-product grading. Similarly, grade $H_{*}(S^{2})$ so that
    the 2-dimensional generator is in degree $1$ and the
    $0$-dimensional generator is in degree $-1$ and grade
    $H_{*}(S^{2})^{\otimes p}$ accordingly. We refer to these
    gradings on both sides as the ``$Q$-grading''. In the case of
    $H_{*}(S^{2})^{\otimes p}$, this grading is the ordinary
    homological grading on the manifold $\Rep(K,\bi)$, shifted down by
    $p$. The isomorphism
    $\epsilon$ in
    \eqref{eq:Morse-to-V} preserves the $Q$-grading modulo $4$
    (essentially because the instanton homology is $\Z/4$
    graded). Furthermore, we can write it as
    \begin{equation}\label{eq:epsilon-terms}
             \epsilon =   \epsilon_{0} + \epsilon_{1}  + \cdots
    \end{equation}
    where $\epsilon_{0}$ preserves the $Q$-grading and $\epsilon_{i}$
    increases the $Q$-grading by $4i$. The term $\epsilon_{0}$ can be
    computed by looking only at flat connections on the excision
    cobordism, and it is not hard to see that $\epsilon_{0}$ is indeed
    the $p$'th tensor power of the standard map. The terms
    $\epsilon_{i}$ for $i$ positive arise from instantons with
    non-zero energy.

    Our conclusion is that the map $\epsilon$ respects the decreasing
    filtration defined by the $Q$-gradings on the two sides, and that
    the induced map on the associated graded objects of these two
    filtrations is standard. 

    \begin{remark}
        In the authors' earlier paper \cite{KM-unknot}, the
        group $V=\langle \vp, \vm \rangle$ appears with a mod-4
        grading in which $\vp$ and $\vm$ have degrees $0$ and $-2$ mod $4$
        respectively. The mod $4$ grading in \cite{KM-unknot} is not
        the same as the grading that we are considering here.
    \end{remark}

    Now let $S$ be a cobordism (not necessarily orientable) 
    from an unlink $K_{1}$ to another
    unlink $K_{0}$. The cobordism $S$ (when equipped with an
    $I$-orientation, to fix the overall sign) induces a map
    \[
               \Isharp(S):  \Isharp(K_{1}) \to \Isharp(K_{0}),
    \]
     or equivalently
    \[
              \Isharp(S) : V^{\otimes b_{0}(K_{1})} \to  V^{\otimes b_{0}(K_{0})}.
    \]
     The $Q$-grading on $V^{\otimes b_{0}(K_{1})}$ and  $V^{\otimes
       b_{0}(K_{0})}$ defines a decreasing filtration on each of them.
     We wish to see what the effect of $\Isharp(S)$ is on this
     $Q$-filtration.

     \begin{lemma}
       For a cobordism   $S$ as above, the induced map $\Isharp(S)$
       has order greater than or equal to
       \begin{equation}\label{eq:Q-order}
                      \chi(S) + S\cdot S  
                         - 4 \Bigl\lfloor \frac{S \cdot S}{8} \Bigr\rfloor
        \end{equation}
       with respect to the filtration defined by $Q$.
     \end{lemma}

     \begin{proof}
         Choose small perturbations $f_{1}$ and $f_{0}$ for the
         Chern-Simons functional on the two ends to achieve the
         Morse-Smale condition; choose them so that all the critical
         points have even index, as in the discussion above.
         Let $\beta_{1}$ and $\beta_{0}$ be critical points for
         $K_{1}^{\sharp}$ and $K_{0}^{\sharp}$ and let
         $M(S;\beta_{1},\beta_{0})$ be the corresponding moduli
         space. The map $\Isharp$ is defined by
         counting points in zero-dimensional moduli spaces of this
         sort; but we consider first the dimension formula in
         general. For each $[A] \in M(S;\beta_{1},\beta_{0})$ we can
         find a nearby configuration $[A']$ which is asymptotic to
         points $\beta'_{1}$ and $\beta'_{0}$ in the critical set of
         the \emph{unperturbed} functional. Let us write
         $\kappa=\kappa(A)$ for the \emph{topological action} of the
         solution $A$, by which we mean the integral
          \[
              \frac{1}{8\pi^{2}} \int \tr ( F_{A'} \wedge F_{A'}).
           \]
          This quantity is a homotopy invariant of $A$, independent of
          the choice of nearby path $A'$. We write
          $M(S;\beta_{1},\beta_{0})$ as a union of parts
          $M_{\kappa}(S;\beta_{1},\beta_{0})$ of different actions $\kappa$.
           
          We  claim that the dimension of
          $M_{\kappa}(S;\beta_{1},\beta_{0})$ is given by the formula
          \begin{equation}\label{eq:dim-simple}
                   \dim M_{\kappa}(S;\beta_{1},\beta_{0})
                            = 8\kappa + \chi(S) + \frac{1}{2} (S\cdot
                            S) + Q(\beta_{1}) - Q(\beta_{0}).
           \end{equation} 
         To verify this, note first that if we change $\beta_{1}$ to a
         different $\beta'_{1}$ while
         keeping $\kappa$ and $\beta_{0}$ unchanged, then the change
         in $\dim M$ is equal to the change in the Morse index of the
         critical points in the representation variety, which is
         $Q(\beta_{1}) - Q(\beta'_{1})$. A similar remark applies to
         $\beta_{0}$, with an opposite sign. So it is enough to check
         the formula for one particular choice of $\beta_{1}$ and $\beta_{0}$.
         So we take $\beta_{1}$ to be the critical point
         of top Morse index, corresponding to the generator
         $\vp\otimes\dots\otimes\vp$, and $\beta_{0}$ to be the critical
         point of lowest Morse index, corresponding to the generator
         $\vm\otimes\dots\otimes\vm$. So $Q(\beta_{1}) = b_{0}(K_{1})$ and
         $Q(\beta_{0}) = - b_{0}(K_{0})$. In this particular case, the
         dimension of $M_{\kappa}(S;\beta_{1},\beta_{0})$ is equal to
         the dimension of 
         \[
                          M_{\kappa}(\bar S; u_{0}, u_{0})
          \]
          where $\bar S$ is the closed surface obtained from $S$ by
          adding disks to all boundary components, regarded as a
          cobordism from the empty link $U_{0}$ to itself, and $u_{0}$
          is the unique critical point in $\bonf^{\omega}(S^{3},
          U_{0}^{\sharp})$ (the generator of $\Isharp(U_{0})=\Z$). The
          dimension in this case can be read off from the dimension
          formula for the case of a closed manifold,
          \cite[Lemma~2.11]{KM-unknot}, which gives
           \begin{equation}\label{eq:M-dim-bar}
                \dim  M_{\kappa}(\bar S; u_{0}, u_{0}) = 8\kappa +
                \chi(\bar S) + \frac{1}{2}(\bar S \cdot \bar S).
            \end{equation}
           Taking account of the added disks, we can write this as
           \begin{equation}\label{eq:M-dim}
            8\kappa +
                \chi( S) + \frac{1}{2}(S \cdot S) + b_{0}(K_{1}) + b_{0}(K_{0}),
           \end{equation}
           which coincides with the formula \eqref{eq:dim-simple} in
           this case. This verifies the formula \eqref{eq:dim-simple}
           for the general case.

          To continue with the proof of the lemma, we make two
          observations about the action $\kappa(A)$:
          \begin{enumerate}
          \item $\kappa(A) = \tfrac{1}{16}(S\cdot S)$ modulo $\tfrac{1}{2}\Z$; and
          \item $\kappa(A)$ is non-negative, as long as the
              perturbations are small.
          \end{enumerate}
          The first of these can be read off from the formula in
          Proposition~2.7 of \cite{KM-unknot}, applied to the closed
          surface $\bar S$ in $\R^{4}$, using the fact that
          $p_{1}(P_{\Delta})$ is divisible by $4$ when $P_{\Delta}$ is
          an $\SU(2)$ bundle. The second of these observations follows
          from the non-negativity of the action for solutions of the
          unperturbed equations. Together, these observations tell us
          that
           \[
                 8\kappa \ge  \frac{1}{2}(S\cdot S  )
                         - 4 \Bigl\lfloor \frac{S \cdot S}{8} \Bigr\rfloor.
           \]
           
          The matrix entries of $\Isharp(S)$ arise from moduli spaces
          of dimension zero; and for such moduli spaces we have
           \[
          \begin{aligned}
              Q(\beta_{0}) - Q(\beta_{1}) &= 8\kappa + \chi(S)
              + \frac{1}{2}(S\cdot S)  \\
              & \ge \chi(S) + S \cdot S - 4 \Bigl\lfloor \frac{S \cdot
                S}{8} \Bigr\rfloor
          \end{aligned}   
         \]
          because of the above inequality for $\kappa$. This last
          quantity is the expression~\eqref{eq:Q-order} in the lemma.
     \end{proof}

      The lemma above is rather artificial (the maps involved are
      often zero in any case), but the method of proof adapts with
      essentially no change to yield a more applicable
      version. We take $K_{1}$, $K_{0}$ and $S$ as above, and we consider a smooth, finite-dimensional family of
      metrics and perturbations on the cobordism, parametrized a by
      manifold $G$. We then have parametrized moduli spaces
      $M(S,\beta_{1},\beta_{0})_{G}$ over the space $G$. Counting
      isolated points in these parametrized moduli spaces gives rise
      to maps
      \[
               m_{G}(S) : V^{\otimes b_{0}(K_{1})} \to  V^{\otimes b_{0}(K_{0})}.
      \]
       Just as in the case above where $G$ is a point, we obtain an
       inequality for $Q$-grading:

       \begin{lemma}
           The map $m_{G}(S)$ has order greater than or equal to
       \[
                        \chi(S) + S\cdot S  
                         - 4 \Bigl\lfloor \frac{S \cdot S}{8}
                         \Bigr\rfloor + \dim G
        \]
       with respect to the decreasing filtration defined by $Q$.
     \end{lemma}

     \begin{proof}
         The proof is unchanged, except that the formula for the
         dimension of the moduli space has an extra term  $\dim G$.
     \end{proof}

     \begin{corollary}
          \label{cor:proto-q-inequality}
         If $S\cdot S < 8$, then $m_{G}(S)$ has order greater than or
         equal to
               \[
                \chi(S) +
                      S\cdot S + \dim G.
        \]
           \qed 
     \end{corollary}

\section{The cube}
\label{sec:cube}

Let $K$ be a link in $\R^{3}$. Figure~\ref{fig:unoriented-skein} shows
three copies of the standard closed ball $B^{3}$, each containing a
pair of arcs: $L_{0}$, $L_{1}$ and $L_{2}$ respectively. By a
\emph{crossing} of $K$ we will mean an embedding of pairs
\[
             c : (B^{3}, L_{2}) \hookrightarrow (\R^{3}, K)
\]
which is orientation-preserving on $B^{3}$ and satisfies $c(L_{2}) =
c(B^{3}) \cap K$.
\begin{figure}
    \centering
    \includegraphics[height=1.5 in]{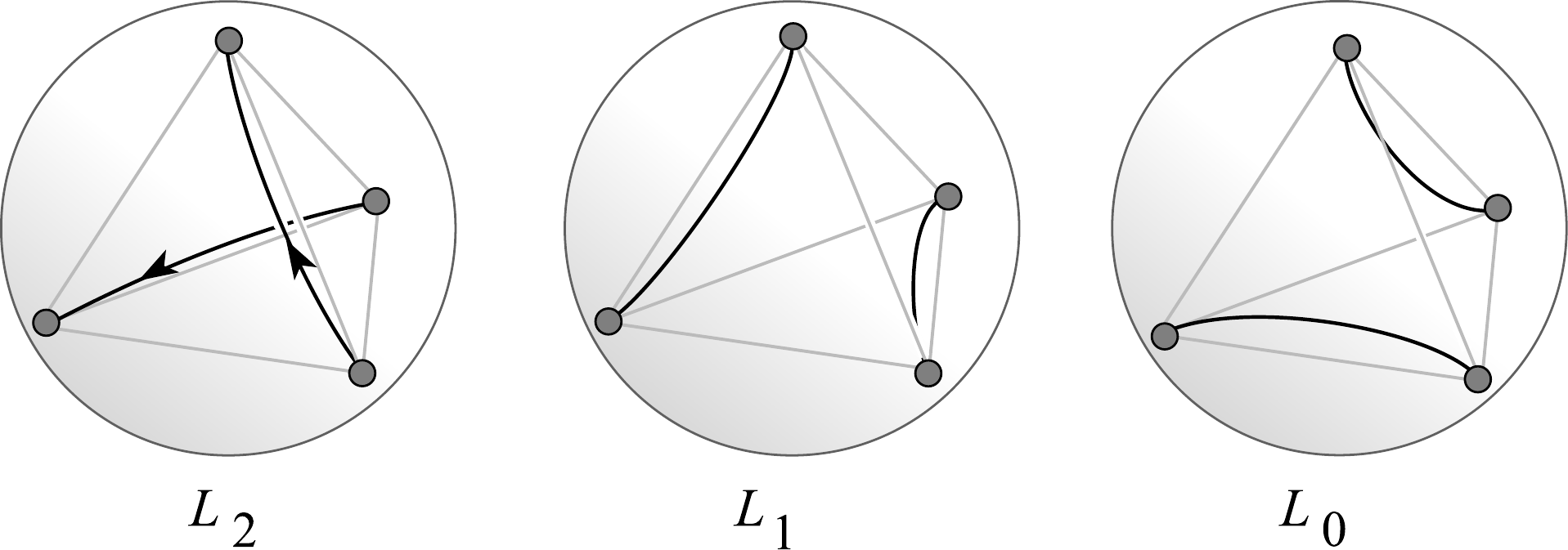}
    \caption{A closed 3-ball, containing a pair of arcs in three different ways.}
    \label{fig:unoriented-skein}
\end{figure}
 The figure also shows a
standard orientation for the pair of arcs $L_{2}\subset B^{3}$. If the
link $K$ is also given an orientation, then we will say that $c$ is a
\emph{positive crossing} if $c : L_{2} \to K$ is either
orientation-preserving or orientation-reversing on both components of
$L_{2}$. Otherwise, if $c$ is orientation-reversing on exactly one
component of $L_{2}$, we say that $c$ is a \emph{negative} crossing.

Let $N$ be a finite set of disjoint crossings
for $K$.  For each $v \in \Z^{N}$,
let $K_{v}\subset \R^{3}$ be the link obtained from $K$ by replacing
$c(L_{2}) \subset K$ by either $c(L_{0})$, $c(L_{1})$ or $c(L_{2})$, 
according to the value of $v(c)$ mod 3, for each crossing
$c \in N$. Thus $K_{v} = K$ in the case that $v : N \to \Z$
is the constant  $2$.

Following the prescription of \cite{KM-unknot}, we choose generic
metrics and holonomy 
perturbations for each link $K_{v}^{\sharp} \subset S^{3}$
so as achieve the Morse-Smale condition. In order to fix signs for the
maps arising later from cobordisms, we also need to choose for each
$v$  a basepoint in $\bonf^{\omega}(S^{3}, K_{v}^{\sharp})$.
We refer to the choice of metric, perturbation and
basepoint as the \emph{auxiliary data} for $K_{v}$. We then have a complex
\[
               (C_{v} , d_{v})
\]
that computes the instanton homology $\Isharp(K_{v})$.

For each $v\ge u$ we have a standard cobordism $S_{vu}$ from $K_{v}$
to $K_{u}$, as in \cite[section 6.1 and 7.2]{KM-unknot}. This
cobordism comes with a family of metrics $G'_{vu}$ defined in
\cite[section 7.2]{KM-unknot}. The dimension of $G'_{vu}$ is $|
v-u|_{1}$ (the sum of the coefficients of $v-u$) and it is acted on by
a 1-dimensional group of translations. The quotient family
$\bG'_{vu}=G'_{vu}/\R$ has dimension $|v-u|_{1}-1$ if $v\ne u$. The
norm $|v-u|_{1}$ is also equal to  $-\chi(S_{vu})$.

\begin{definition}
    We say that a cobordism $S_{vu}$ (or sometimes, a pair $(v,u)$)
    with $v,  u \in \Z^{N}$ is of \emph{type $n$} for $n \ge 0$ if
    $ v\ge u$ and
    \[
           \max \{ \, v(c) - u(c) \mid c\in N \, \} = n.
     \]
    In particular, $(v,u)$ has \emph{type $0$} if and only if $v=u$.
\end{definition}

In the case that $S_{vu}$ has type $1$, $2$ or $3$, the authors
defined in \cite{KM-unknot} a larger family of metrics, $\bG_{vu}$
containing $\bG'_{vu}$. In the case of type $1$, the space $\bG_{vu}$
coincided with $\bG'_{vu}$; for type $2$ and $|N|=1$, the inclusion of $\bG'_{vu}$
in $\bG_{vu}$ was the inclusion of a half-line in  $\R$, and for type $3$ it
was the inclusion of a ``quadrant'' in an open pentagon. In all cases, the dimensions of
$\bG_{vu}$ and $\bG'_{vu}$ are equal.

If we choose an $I$-orientation for each cobordism $S_{vu}$ and an
orientation for the family of metrics $\bG'_{vu}$ (or equivalently the
family $\bG_{vu}$, when defined), then we have oriented, parametrized
moduli spaces of instantons,
\[
          M_{vu}(\beta,\alpha) \to \bG_{vu}
\]
for each pair of critical points $\beta\in \Crit_{v}$ and $\alpha \in
\Crit_{u}$, whenever the pair $(v,u)$ has type $3$ or
less. Consistency conditions are imposed on the chosen orientations:
see for example Lemmas 6.1 and 6.2 in \cite{KM-unknot}. In addition to
the auxiliary data for each, secondary
perturbations on the cobordisms must be chosen, to ensure that the
moduli spaces are regular.
By counting points in zero-dimensional parametrized moduli spaces, we
obtain maps between the corresponding groups $C_{v}$ and $C_{u}$.
Following the notation of \cite{KM-unknot} we write these maps as
\[
        \breve m_{vu} : C_{v} \to C_{u}.
\]
The orientation conventions which are specified in
\cite{KM-unknot} lead to extra signs in the various gluing formulae,
so it is convenient to introduce the following variant: we define
\[
                 f_{vu} : C_{v} \to C_{u}
\]
by the formula
\[
                 f_{vu} = (-1)^{\msign(v,u)} \breve m_{vu}
\]
where
\begin{equation}\label{eq:m-formula}
         \msign(v,u) =  \frac{1}{2}|v-u|_{1}(|v-u|_{1}-1) + \sum_{c\in
         N} v(c).
\end{equation}
(In \cite{KM-unknot} the notation $f_{vu}$ was reserved for the case
of type $0$ or $1$, and the notation $j_{vu}$ or $k_{vu}$ was used
for type $2$ or $3$. For efficiency however, we here adopt $f_{vu}$
for all these cases.) It is also convenient to define $\breve m_{vv} =
d_{v}$ for the case that $v=u$ -- i.e. the case of type $0$ -- so that
\[
                 f_{vv} = (-1)^{\sum v(c)} d_{v}.
\]

Some chain-homotopy formulae involving these maps are proved in
\cite{KM-unknot}. For $(v,u)$ a pair of type $0$, $1$ or $2$, the
formulae all  take the same basic form, given in the following
proposition.

\begin{proposition}\label{prop:generic-fsquared}
    For $(v,u)$ of type $n\le 2$, we have
    \[
    \sum_{\substack{w \\ v\ge w \ge u}} 
    f_{wu} f_{vw} =0 .
     \]
   \qed
\end{proposition}

(There is a also a formula in \cite{KM-unknot} for the case of type
$3$, but this involves additional terms: see
\eqref{eq:type-3-identity} in section~\ref{sec:dropping-crossing}.)  In
the case of type $0$, so that $v=u$, the formula in the proposition
says $d_{v}^{2}=0$, expressing the fact that $d_{v}$ is a differential.

We write
\[
         \bC(N) = \bigoplus_{v: N \to \{0,1\}} C_{v}.
\]
This is a sum of the complexes indexed by the vertices of a cube of
dimension $|N|$. We write $\bF = \bF(N)$ for the map
\[
           \bF : \bC(N) \to \bC(N)
\]
given by
\[
              \bF = \bigoplus_{
                     \substack{
                     v,u : N \to \{0,1\} \\
                      v\ge u
                     } }   f_{vu}.
\]
Note that the summands $f_{vu}$ in this definition all have type $0$
or $1$. Proposition~\ref{prop:generic-fsquared} tells us that $\bF^{2}=0$, so
we have a complex. 

We have had to choose auxiliary
data for each $K_{v}$, secondary perturbations for the moduli spaces associated to
the cobordisms $S_{vu}$, as well as consistent $I$-orientations for the cobordisms
and orientations for the families of metrics $\bG'_{vu}$. We refer to this collection
of choices as \emph{auxiliary data for $(K,N)$}.

The following is proved in \cite{KM-unknot}.

\begin{theorem}\label{thm:any-cube}
    For any two collections of crossings, $N$ and $N'$, and any
    corresponding choices of auxiliary data,
    the complexes $(\bC(N), \bF(N))$ and
    $(\bC(N'), \bF(N'))$ are quasi-isomorphic. \qed
\end{theorem}

Of course, it is sufficient to deal with the case that $N'$ is
obtained from $N$ by forgetting just one crossing; and this is how the
proof is given in \cite{KM-unknot} (see Proposition~6.11 of that paper).
We will later refine this theorem, replacing ``quasi-isomorphic''
with ``chain-homotopy equivalent.'' As a special case we can take $N'$
to be empty, and we obtain:

\begin{corollary}[{\cite[Theorem~6.8]{KM-unknot}}]
 \label{cor:cube-iso}
    For any collection $N$ of crossings of $K$, 
   the homology of the complex $(\bC(N), \bF(N))$ is isomorphic to $\Isharp(K)$.
\end{corollary}

\section{The $q$-filtration on cubes}
\label{sec:q-cube}

We continue to consider $K\subset \R^{3}$ with a collection $N$ of
crossings, and the complex $(\bC(N), \bF(N))$ defined in the previous
subsection. We will suppose that the collection $N$ has the following
property.

\begin{definition}\label{def:pseudo-diagram}
    We will say that a link $K$ with a collection $N$ of crossings
    is a \emph{pseudo-diagram} if, 
    for all $v : N\to \{0,1\}$, the link $K_{v} \subset \R^{3}$ is an
    unlink. In this case, we refer to the unlinks $K_{v}$ as the
    \emph{resolutions} of $K$.
\end{definition}

 As in section~\ref{subsec:unlinks}, whenever $K_{v}$ is an unlink, we can choose
the auxiliary data so that the
corresponding differential $d_{v}$ is zero, in which case $C_{v}$ can
be identified with the homology of the representation variety,
$\Rep(K_{v},\bi) = (S^{2})^{p(v)}$ by the isomorphism $\beta$ of
\eqref{eq:rep-homology}. 
When this is done, we say that we
have chosen \emph{good auxiliary data} for $(K,N)$.

The terminology in Definition~\ref{def:pseudo-diagram} is chosen
because the condition holds when $N$ is the set of crossings that
arises from a plane diagram of $K$. But the case of a plane
diagram is special in other ways: for example, the cobordisms
$S_{vu}$, for $v, u: N \to \{ 0,1\}$, are always orientable if $N$
arises from a diagram, whereas Definition~\ref{def:pseudo-diagram}
certainly allows some $S_{vu}$ to be non-orientable. In particular,
the self-intersection numbers $S_{vu} \cdot S_{vu}$ may be
non-zero. For $v\ge u$ we define
\[
         \sig(v,u) = S_{vu} \cdot S_{vu}.
\]
In the case that $w \ge v \ge u$, we have $\sig(w,v) + \sig(v,u) = \sig(w,u)$,
so we can consistently define $\sig(v,u)$ even when we do \emph{not} have
$v\ge u$ by insisting on the additivity property. Thus, for example,
$\sig(v,u) = - \sig(u,v)$. We can extend this notation beyond the cube
$\{0,1\}^{N}$ to all elements $v \in \Z^{N}$ with the property that
$K_{v}$ is an unlink. Thus, if $v$, $u$ both have this property we may
consistently define
\[
          \sig(v,u) = S_{vw} \cdot S_{vw}  - S_{uw} \cdot S_{uw}
\]
where $w$ is any chosen third point with $v\ge w$ and $u \ge w$. 

\begin{lemma}\label{lem:3-step-sig}
    Suppose $v\in \Z^{N}$ is such that $K_{v}$ is an unlink, and
    suppose that $v-u$ is divisible by $3$, so that $K_{u} \cong
    K_{v}$ is also an unlink. Then \[
           \sig(v,u) = \frac{2}{3}\sum_{c}(v(c)-u(c)).
    \]
\end{lemma}

\begin{proof}
    It is enough to consider only the case that $v$ and $u$ differ at
    only a single crossing $c_{*}$, with $v(c_{*}) - u(c_{*})=3$. In
    this case, the cobordism $S_{vu}$ is a composite of three
    cobordisms, $S_{vv'}$, $S_{v'v''}$ and $S_{v''u}$, with
    $v'(c_{*})=v(c)-1$ and $v''(c_{*}) = v(c)-2$. As explained in
    \cite{KM-unknot}, the cobordism $S_{vv''}$ (the composite of the
    first two) can be
    described as a connected sum of the opposite of $S_{v''u}$ with
    standard pair $(S^{4}, \RP^{2})$, where the $\RP^{2}$ has
    self-intersection $2$ in $S^{4}$. So $S_{vu}$ is obtained from the
    composite of $S_{v''u}$ with its opposite, by  summing with this
    $\RP^{2}$. So $S_{vu}\cdot S_{vu}=2$. Thus $\sig(v,u) =
    (2/3)(v(c_{*}) - u(c_{*}))$ in this case.
\end{proof}

Suppose now that $(K,N)$ is a pseudo-diagram, and let us choose good
auxiliary data.
As in section~\ref{subsec:unlinks} we obtain
an isomorphism
\[
       \beta: \Isharp(K_{v}) \to  V^{\otimes b_{0}(K_{v})}
\]
via the identifications
\begin{equation}\label{eq:Kv-id}
\begin{aligned}
    \Isharp(K_{v}) &= C_{v} \\
    &\cong H_{2}(S^{2})^{\otimes b_{0}(K_{v})} \\
    &= V^{\otimes b_{0}(K_{v})}
\end{aligned}
\end{equation} 
because the differential $d_{v}$ is zero.
As before, we give $V^{\otimes b_{0}(K_{v})}$ a grading $Q$, 
by declaring that the generators $\vp$ and $\vm$ in $V$ have
$Q$-grading $1$ and $-1$.

We define the $q$-grading on $C_{v}$ by shifting the $Q$-grading
by some correction terms. Choose first an orientation
for $K$. At each crossing $c\in N$, one of the resolutions $0$ or $1$
is preferred as the \emph{oriented} resolution. We write $o:
N\to\{0,1\}$ for the function that assigns to each crossing its
oriented resolution. Thus $K_{o}$ can be oriented in such a way that
its orientation agrees with the original orientation of $K$ outside
the crossing-balls. For $v : N \to \{ 0, 1\} $ we then set
\begin{equation}\label{eq:q-def}
                q = Q - \left( \sum_{c\in N} v(c) \right) 
          + \frac{3}{2}\sigma(v,o)  -
                n_{+} + 2 n_{-}
\end{equation}
on $C_{v}$, where $n_{+}$ and $n_{-}$ are the number of positive and
negative crossings respectively, so that
\[
           n_{+} + n_{-} = |N|.
\]
With the exception of the self-intersection term $\sigma(v,o)$ 
(which is zero in the
case arising from a plane diagram), these correction terms are
essentially the same as those presented by Khovanov in
\cite{Khovanov}. The $q$-gradings on all the
vertices of the cube gives us a grading $q$ on $\bC(N)$. Note that the
correction terms in the formula above do not depend on a choice of
orientation for $K$ if $K$ is a knot rather than a link.

We can also define $q$ on $C_{v}$ when $v$ more generally is in
$\Z^{N}$ rather than $\{0,1\}^{N}$ subject to the constraint that
$K_{v}$ is an unlink: we use the same formula. Then we have:

\begin{lemma}\label{lem:q-is-periodic}
    Suppose $v$ is such that $K_{v}$ is an unlink, and let $u$ be such
    that $v-u$ is divisible by
    $3$ in $\Z^{N}$, so that $K_{v} = K_{u}$. Identify $C_{v}$
    with $C_{u}$ as abelian groups, via the isomorphisms $\beta$ (see
    \eqref{eq:Kv-id}). 
   Then
    the $q$-gradings on $C_{v}$ and $C_{u}$ coincide.
\end{lemma}

\begin{proof}
    In the definition of $q$, the terms $Q$, $n_{+}$ and $n_{-}$ are
    unchanged on replacing $v$ by $u$. The remaining terms are 
    $-\sum v(c)$ and
    $(3/2) \sum \sigma(v,o)$, and the changes in
    these terms cancel,  as an
    immediate consequence of Lemma~\ref{lem:3-step-sig}.
\end{proof}

We also note:

\begin{lemma}
    The parity of $q$ on $\bC(N)$ is constant, and is equal to the
    number of components of $K$ mod $2$.
\end{lemma}

\begin{proof}
   At each $v\in\{0,1\}^{N}$, the parity of $Q$ on $C_{v}$ is equal to
   the number of components of $K_{v}$. (This follows immediately from
   the definition.) So it is clear that the parity of $q$ is constant
   on each $C_{v}$. We must check that its parity is independent of
   $v$. For this we consider two adjacent vertices $v$, $v'$ in the cube
   $\{0,1\}^{N}$. The term $\sum v(c)$ which appears in the
   definition of $q$ then changes by $1$ between $v$ and $v'$. The
   change in the term $(3/2)\sigma(v,o)$ is equal to
   $(1/2)\sigma(v,v')$ mod $2$, which is zero if the cobordism
   $S_{vv'}$ is orientable and is equal to $1$ if it is
   non-orientable, because the self-intersection number of an
   $\RP^{2}$ in $\R^{4}$ is equal to $2$ mod $4$. In the orientable case,
   the number of components of $K_{v}$ and $K_{v'}$ differ by $1$, so
   the parity of $Q$ changes by $1$. In the non-orientable case, the
   parity of $Q$ is unchanged. Altogether, exactly two of the first
   three terms in the definition of $q$ change parity. So the parity
   of $q$ is indeed constant.

   Now let $K_{o}$ denote the oriented resolution of our original knot
   $K$. To obtain the oriented resolution, we must set $o(c)=0$ at
   every positive crossing and $o(c)=1$ at every negative crossing. So
   when $v=o$, we have
   \[
              \sum_{v\in N}v(c) = n_{-}.
    \]
   At this vertex of the cube, we therefore have
   \[
   \begin{aligned}
       q &= b_{0}(K_{o}) - n_{-} - n_{+} \\
       &= b_{0}(K_{o}) - |N|
   \end{aligned}
   \]
  mod $2$. The cobordism from $K$ to $K_{o}$ is orientable, and
  is obtained by attaching $|N|$ $1$-handles. As above, each $1$-handle
  addition changes the number of components by $1$. So $b_{0}(K_{o})-|N|
  = b_{0}(K)$ mod $2$. This concludes the proof of the lemma.
\end{proof}

Although we can define the $q$-grading on $\bC(N)$ whenever $(K,N)$ is
a pseudo-diagram, it is not the case (a priori, at least) that the
differential $\bF(N) : \bC(N) \to \bC(N)$ respects the decreasing filtration
defined by $q$. For this, we need an additional condition.

\begin{definition}
    \label{def:small-squares}
We say that $(K,N)$ has \emph{small self-intersection
numbers} 
if $S_{vu} \cdot S_{vu} \le 6$ for all $v\ge u$ in $\{0,1\}^{N}$.
\end{definition}

\begin{proposition}\label{prop:C-is-filtered}
    If $(K,N)$ is a pseudo-diagram with small self-intersection
    numbers, 
   then the differential $\bF(N)
    : \bC(N) \to \bC(N)$ has order $\ge 0$ with respect to 
     the decreasing filtration defined by $q$.
\end{proposition}

\begin{proof}
    The map $\bF = \bF(N)$ is the sum of the maps $f_{vu}$, each of
    which is obtained by counting instantons on a cobordism $S_{vu}$
    over a family of metrics of dimension $-\chi(S_{vu})-1$. Because the
    self-intersection number of $S_{vu}$ is at most $6$, we can apply
    Corollary~\ref{cor:proto-q-inequality} to the map $f_{vu}$. That
    corollary tells us that, with respect to the decreasing filtration
    defined by $Q$ on on $C_{v}$ and $C_{u}$, the map $f_{vu}$ has
    order $\ge S_{vu}\cdot S_{vu} -1$.
    If we instead consider the decreasing filtration $\cF^{j}$ defined
    by $q$ instead of $Q$, then we obtain
    \begin{equation}\label{eq:q-f-order}
         \ord_{q} f_{vu} \ge 
             -\frac{1}{2} ( S_{vu} \cdot S_{vu}) + \sum_{c} ( v(c) -
          u(c) ) - 1.
   \end{equation}
   Now we need:

   \begin{lemma}\label{lem:max-is-2}
      If $(K,N)$ is a pseudo-diagram, we have
    \[
        ( S_{vu} \cdot S_{vu})  \le 2 \sum_{c} ( v(c) -
          u(c) ),
    \] 
     for all pairs $v\ge u$ in $\{0,1\}^{N}$.
   \end{lemma}

   \begin{proof}
    This can immediately be reduced to the case that $v$ and $u$
    differ at only one crossing $c$. In that case, $S_{vu}$ is the
    union of some cylinders (corresponding to components of $K_{v}$
    that do not pass through the crossing) and a single non-trivial
    piece that is either a pair of pants or a twice-punctured
    $\RP^{2}$. The links $K_{v}$ and $K_{u}$ are trivial by
    hypothesis, and the self-intersection number $S_{vu}\cdot S_{vu}$
    is equal to that self-intersection of the closed surface obtained
    from the cobordism by adding disks. Thus the self-intersection
    number is equal to that of either a surface that is either a union
    of spheres only, or a union of spheres with a single $\RP^{2}$. In
    the former case, the self-intersection is zero. In the latter
    case, the self-intersection of an $\RP^{2}$ in $\R^{4}$ is equal
    to either $2$ or $-2$ \cite{Massey}. In all cases, the
    self-intersection number is no larger than $2$.
   \end{proof}

    To continue with the proof of the Proposition, it follows from the
    lemma now that the right-hand side of \eqref{eq:q-f-order} is $\ge
    -1$, and so $f_{vu} (\cF^{j})
    \subset \cF^{j -1}$ for all $j$. However, since the $q$-grading
    takes values of only one parity, it is in fact the case that
    $f_{vu} (\cF^{j}) \subset \cF^{j}$ for all $j$.
\end{proof}

\section{Isotopy and ordering} 
\label{sec:isotopy}

Given a link $K$ with a set of crossings $N$ and choices of auxiliary
data, we have constructed in the previous section a $q$-filtered complex
$\bC(K,N)$ with its differential $\bF(K,N)$. We now begin the task of
investigating to what extent this filtered complex is dependent on the
choices made. For reference, let us introduce the category $\hFD$ whose objects are
filtered, finitely-generated abelian groups with a
differential of order $\ge 0$, and whose morphisms are
chain maps of order $\ge 0$ modulo 
chain-homotopies of order $\ge 0$.
(As elsewhere in this paper we
continue to refer to a differential group as a ``chain complex'' even
when there is no $\Z$-grading. All our differential groups will be at
least $\Z/4$-graded.) Our complex $\bC(K,N)$ is an element of $\hFD$.

We consider what happens when we keep the
crossings the same but change $K$ by an isotopy and change the choice
of auxiliary data. That is, we fix a set of crossings
$N$ for $K$, and
we suppose that
$K'$ is isotopic to $K$ by an isotopy that is constant inside the
union 
\[
        B(N) = \bigcup_{c\in N} c(B^{3}).
\]
 Thus the trace of this isotopy is a
cobordism $T$ from $K$ to $K'$ that is topologically a cylinder in $\R\times\R^{3}$ and is
metrically a cylinder inside $\R\times B(N)$. We choose
auxiliary data for both $(K,N)$ and $(K',N)$, giving rise to chain
complexes
\[
   \bC(K, N), \quad \bC(K', N)
\]
constructed from the cubes of resolutions.

From the cobordism $T$ we obtain cobordisms $T_{vv}$ from $K_{v}$ to
$K'_{v}$ for all $v: N \to \{0,1\}$. For $v\ge u$, we also obtain a
cobordism $T_{vu}$ from $K_{v}$ to $K'_{u}$ equipped with a family of
metrics $H_{vu}$ of dimension $\sum (v(c) - u(c))$: these cobordisms
and metrics are all the same outside $B(N)$, while inside $B(N)$ they
coincide with the metrics $G_{vu}$ with which the cobordisms $S_{vu}$
were earlier equipped. By counting instantons over the cobordisms
$T_{vu}$ equipped with these metrics and generic secondary
perturbations, 
we obtain a chain map of the
cube complexes,
\[
           \MATHBF{T} : \bC(K, N) \to \bC(K', N).
\]
By the usual sorts of arguments, two different choices of metrics and
secondary perturbations on
the interior of the cobordism give rise to chain maps $\MATHBF{T}$
that differ by chain homotopy, and it follows that $\bC(K,N)$ and
$\bC(K',N)$ are chain-homotopy equivalent.

Now we introduce $q$-gradings. For this we suppose that $(K,N)$ is a pseudo-diagram with small
self-intersection numbers, as in
Definition~\ref{def:small-squares}. Of course,  $(K',N)$ then shares
these properties. For $(K,N)$ and $(K',N)$ we ensure that our chosen
auxiliary data is good.
The complexes $\bC(K,N)$ and $\bC(K', N)$  both then have
$q$-filtrations preserved by the differential
(Proposition~\ref{prop:C-is-filtered}). 
Arguing as in the proof of Proposition~\ref{prop:C-is-filtered}, 
we see that the chain map $\MATHBF{T}$ is
filtration-preserving, as are the chain-homotopies between different
chain maps $ \MATHBF{T}$  and $\tilde{\MATHBF{T}}$ arising from
different choices of metrics and secondary perturbations on the
cobordisms. (The dimensions of the families of metrics on
$T_{vu}$ are larger by $1$ than those on $S_{vu}$, but this only helps
us. The proofs are otherwise the same.)  Thus,

\begin{proposition}\label{prop:isotopy-cube}
    Let $(K,N)$ be a pseudo-diagram with small self-intersection
    numbers, and suppose we have an isotopy from $K$ to $K'$ relative
    to $B(N)$. Let $\bC(K,N)$, $\bC(K',N)$ be the $q$-filtered
    complexes obtained from $(K,N)$ and $(K', N)$ via choices of good
    auxiliary data. Then the isotopy gives rise to a
    well-defined isomorphism $\bC(K,N) \to \bC(K',N)$ in the category
    $\hFD$. \qed
\end{proposition}

\begin{remark}
    As usual in Floer theory, a special case of the above proposition
    is the case that $K=K'$ and the isotopy is trivial, in which case
    the objects $\bC(K,N)$ and $\bC(K',N)$ differ only in that the
    auxiliary data may be chosen differently.
\end{remark}

We now have a diagram of maps on homology groups:
   \begin{equation}\label{eq:square-commutes}
              \xymatrix{
                          H(\bC(K,N)) \ar[r]^{\MATHBF{T}_{*}} \ar[d]  &   H(\bC(K',N)) \ar[d] \\
                          \Isharp(K)     \ar[r]^{\Isharp(T)} &   \Isharp(K') 
	   }
   \end{equation}
The maps $\Isharp(T)$ is the isomorphism induced by a cylindrical
cobordism in the usual way, and $\MATHBF{T}_{*}$ is induced by the
chain map $\bT$. The vertical maps are the isomorphisms of
Corollary~\ref{cor:cube-iso}.  It is a straightforward application of
the usual ideas, to show that this diagram commutes. To do this, we
remember that the isomorphisms of Corollary~\ref{cor:cube-iso} are
obtained as a composite of maps, forgetting the crossings of $N$ one
by one. Thus one should only check the commutativity of a digram such
as this one, where $N' = N \setminus\{c_{*}\}$:
   \[
              \xymatrix{
                          H(\bC(K,N)) \ar[r]^{\MATHBF{T}(N)_{*}} \ar[d]  &   H(\bC(K',N)) \ar[d] \\
                            H(\bC(K,N'))   \ar[r]^{\MATHBF{T}(N')_{*}} &    H(\bC(K',N')).
	   }
   \]
A chain homotopy between  the composite chain-maps at the level of these cubes is
constructed by counting instantons on cobordisms $T_{vu}$ with $v$ and
$u$ in the appropriate cubes, to obtain a map
\[
                \bC(K,N) \to \bC(K', N').
\]

There is an additional point concerning the commutativity of the
square \eqref{eq:square-commutes}. 
Corollary~\ref{cor:cube-iso} provides the isomorphisms
which are the vertical arrows in the diagram, but the construction of
these isomorphisms depends on a choice of ordering for the set of
crossings $N$, because the map is defined by ``removing one crossing
at a time''. Despite this apparent dependence, the isomorphism $H_{*}(\bC(N)) \to
   \Isharp(K)$ is actually independent of the ordering, up to overall
   sign. To verify this, it is of course enough to consider the effect of
  changing the order of just two crossings which
   are adjacent in the original ordering of $N$.  The
   chain-homotopy formula that we need in this situation is another
   application of Proposition~\ref{prop:generic-fsquared}. To see
   this, consider for example the case that $c_{1}$ and $c_{2}$ are
   the first two crossings in a given ordering of $N$. 
   The construction  of the isomorphism in Corollary~\ref{cor:cube-iso}
   begins with the composite of two chain-maps
    \[
         \bC(K,N) \stackrel{p}{\to} \bC(K, N\sminus\{c_{1}\})  \stackrel{q}{\to}
       \bC(K,N\sminus\{c_{1}, c_{2}\}).
    \]
   If we switch the order, we have a different composite, going via a
   different middle stage:
    \[
       \bC(K,N)  \stackrel{r}{\to} \bC(K, N\sminus\{c_{2}\})  \stackrel{s}{\to}
       \bC(K,N\sminus\{c_{1}, c_{2}\}).
    \]
    The induced maps in homology are the same up to sign, because $qp$
    is chain-homotopic to $-sr$. The chain homotopy is the map
    $\bC(K,N) \to \bC(K,N\sminus\{c_{1}, c_{2}\})$ obtained as the sum
    of appropriate terms $f_{vu}$.     

    Via the isomorphism of Corollary~\ref{cor:cube-iso}, the group
    $\Isharp(K)$ obtains from $\bC(K,N)$ a decreasing
    $q$-filtration. We can interpret the above results as telling us
    that this filtration is independent of some of the choices
    involved:

\begin{proposition}\label{prop:q-filtration-natural}
    Let $K$ be a link in $\R^{3}$, and let $N$ be a collection of
    crossings such that $(K,N)$ is a pseudo-diagram with small
    self-intersection numbers. Then, via
    the isomorphism with $H(\bC(K,N))$, the instanton homology group
    $\Isharp(K)$ obtains a quantum filtration that does not depend on
    an ordering of $N$. Furthermore, this filtration of $\Isharp(K)$
    is natural for 
    isotopies of $K$ rel $B(N)$. \qed
\end{proposition}

\section{Dropping a crossing}
\label{sec:dropping-crossing}

We continue to suppose that $(K,N)$ is a pseudo-diagram with small
self-intersection numbers. 
The isomorphism class in $\hFD$ of the filtered complex
obtained from the cube of resolutions is
independent of the choice of good auxiliary data by the previous
proposition, but we must address the question of whether it depends on $N$.

We begin investigating this question by considering 
the situation in which $N' \subset N$ is obtained by
dropping a single crossing. We suppose that good auxiliary data is chosen
for both $(K,N)$ and $(K,N')$ so that we have complexes $\bC(N)$ and
$\bC(N')$. For the moment, we do not consider the $q$-filtrations on
these. Let us recall from \cite{KM-unknot} the proof
that the homologies of the two cubes $\bC(N)$ and $\bC(N')$ 
are isomorphic (the basic case
of Theorem~\ref{thm:any-cube}). Let $c_{*}\in N$ be the distinguished
crossing that does not belong to $N'$. Write
\[
          \bC(N) = \bC_{1} \oplus \bC_{0}
\]
where
\begin{equation}
    \label{eq:bCi-def}
               \bC_{i} = \bigoplus_{\substack{v(c_{*})=i\\
                   v(c')\in\{0,1\}\;\text{if}\; c'\ne c_{*} }} C_{v}.
\end{equation}
The differential $\bF(N)$ on $\bC(N)$ then takes the form
\begin{equation}\label{eq:F-matrix}
            \bF(N) = \begin{bmatrix} \bF_{11} & 0  \\
                                     \bF_{10} & \bF_{00} 
                      \end{bmatrix}.
\end{equation}
We extend the notation $\bC_{i}$ as just defined to all $i$ in $\Z$.
Whenever $i>j$ and $|i-j|=n \le 3$, we have a map
\[
           \bF_{ij} : \bC_{i} \to \bC_{j}
\]
given as the sum of maps $f_{vu}$, where each pair $(v,u)$ has type
$n$. Similarly, we have $\bF_{ii} : \bC_{i} \to \bC_{i}$, which is a
sum of maps indexed by pairs of type $0$ and $1$.

The 3-step periodicity means that the complexes $\bC_{2}$ and
$\bC_{-1}$ are obtained from the same $(|N|-1)$-dimensional cube of
resolutions of $K$. But the complexes $(\bC_{2}, \bF_{2,2})$ and
$(\bC_{-1},\bF_{-1,-1})$ may differ because of the different choices
of auxiliary data. (We are free to arrange that the choices of metrics,
perturbations and base-points so as to respect the periodicity, but not
the choices involved in orienting the moduli spaces.) Nevertheless,
because the links $K_{v}$ that are involved are the same, there are
canonical cylindrical cobordisms which give rise to a chain-homotopy
equivalences
\begin{equation}\label{eq:T-3step}
            \bT_{2,-1}: (\bC_{2}, \bF_{2,2}) \to (\bC_{-1}, \bF_{-1,-1}).
\end{equation}
as in the previous section. Indeed, both of these chain complexes 
are canonically chain-homotopy equivalent to
$(\bC(N'), \pm \bF(N'))$. Thus,
the summand $C_{v'} \subset \bC(N')$ corresponding to $v' : N' \to
\{0,1\}$ is identified with the summand $C_{v} \subset \bC_{2}$, where
$v: N\to\Z$ is obtained by extending $v'$ to $N$ with $v(c_{*})=2$,
and the cylindrical cobordisms give a chain-homotopy equivalence
\[
        (\bC(N'), \bF(N')) \to (\bC_{2}, \bF_{2,2}).
\]

To
show that $\bC(N)$ and $\bC(N')$ are homotopy-equivalent, we therefore
need to provide an  equivalence
\begin{equation}\label{eq:Psi-Phi-1}
\begin{gathered}
   \Psi: \bC_{1} \oplus \bC_{0} \to \bC_{-1}.
\end{gathered}
\end{equation}
This is  done in \cite{KM-unknot}, where it is shown
that such a chain-homotopy equivalence is provided by the map
\begin{equation}\label{eq:Psi-def}
                 \Psi =     \begin{bmatrix} \bF_{1,-1} & \bF_{0,-1} \end{bmatrix}.
\end{equation}
We recall the proof from \cite{KM-unknot} that this map is
invertible. We consider the maps
\begin{equation}\label{eq:Psi-Phi-2}
\begin{gathered}
  \Phi_{2}:   \bC_{2} \to \bC_{1} \oplus \bC_{0}.\\ 
     \Phi_{-1}:   \bC_{-1} \to \bC_{-2} \oplus \bC_{-3}
\end{gathered}
\end{equation}
given by
\begin{equation}\label{eq:Phi-def}
                   \Phi_{2} =  \begin{bmatrix} \bF_{21}  \\ \bF_{20} \end{bmatrix}
\end{equation}
with $\Phi_{-1}$ defined similarly, shifting the indices down by $3$.
We will show that both composites 
 $\Psi\comp\Phi_{2}$ and $\Phi_{-1}\comp\Psi$ are chain-homotopy
 equivalences, from which it will follow that $\Psi$ is a
 chain-homotopy equivalence.

The composite chain-map $\Psi\comp\Phi_{2}$ is the
map
\begin{equation}\label{eq:first-chain-composite}
\begin{aligned}
    ( \bF_{1,-1} \bF_{21} + \bF_{0,-1} \bF_{20}) : \bC_{2} &\to
    \bC_{-1}.
\end{aligned}
\end{equation}
A version of Proposition~\ref{prop:generic-fsquared} for type $3$
cobordisms (essentially equation (43) of \cite{KM-unknot})  gives an identity
\begin{equation}\label{eq:type-3-identity}
     \bF_{2,-1}\bF_{2,2}+\bF_{1,-1}\bF_{2,1}+\bF_{0,-1}\bF_{2,0}+\bF_{-1,-1}\bF_{2,-1} 
= \bT_{2,-1} + \bN_{2,-1}
\end{equation}
for a certain map $\bN_{2,-1}$. Here $\bT_{2,-1}$ is again the chain
map $\bC_{2} \to \bC_{-1}$ arising from the cylindrical cobordisms.\footnote[1]{In
\cite{KM-unknot}, the authors wrote this as $\pm \mathbf{Id}$, assuming that
the perturbations and orientations were chosen so that  $\bC_{2}$ and
$\bC_{-1}$ were the same complex. In fact, it may not be possible to
choose the signs consistently so that this is the case: the best one
can do is arrange that $\bT_{2,-1}$ is diagonal in a standard basis
with diagonal entries $\pm 1$. This point does not affect the
arguments of \cite{KM-unknot}.}
We can interpret the above equation as saying that the map
\eqref{eq:first-chain-composite} is chain-homotopic to ${\bT}_{2,-1} +
\bN_{2,-1}$, and that the chain homotopy is the map
$\bF_{2,-1}$. Since $\bT_{2,-1}$ is an equivalence, it
then remains to show that $\bN_{2,-1}$ is chain-homotopic to zero. In
\cite{KM-unknot} this null-homotopy is exhibited as a map
\begin{equation}\label{eq:chain-H22}
   \bH_{2,-1} : \bC_{2} \to \bC_{-1}
\end{equation} 
whose matrix entries $h_{vu}$ are defined by counting instantons over a $S_{vu}$ for
a family of metrics of dimension \[\sum_{c\ne c_{*}}(v(c) - u(c)). \]

For the other composite  $\Phi_{-1}\comp\Psi$ the story is similar. We
may write the composite as
\[\Phi_{-1}\comp\Psi =
\begin{bmatrix}
    \bF_{-1,-2}\bF_{1,-1} & \bF_{-1,-2} \bF_{0,-1} \\ 
    \bF_{-1,-3}\bF_{1,-1} & \bF_{-1,-3} \bF_{0,-1}
\end{bmatrix}.
\]
We modify this by  the chain-homotopy
\begin{equation}\label{eq:chain-L}
             \MATHBF{L} = \begin{bmatrix}
    \bF_{1,-2} & \bF_{0,-2} \\ 
    \bF_{1,-3} & \bF_{0,-3}
\end{bmatrix}
\end{equation}
and see that  $\Phi_{-1}\comp\Psi$ is chain-homotopic to a map of the form
\[
\begin{bmatrix}
     \bT_{1,-2} + \bN_{1,-2} & 0 \\
     \MATHBF{Y} & \bT_{0,-3} + \bN_{0,-3}
\end{bmatrix} : \bC_{1} \oplus \bC_{0} \to  \bC_{-2} \oplus \bC_{-3}.
\]
Here $\bN_{1,-2}$ and $\bN_{0,-3}$ are similar to $\bN_{2,-1}$
above, and $\MATHBF{Y}$ is an unidentified term involving cobordisms
of type $4$. If we apply a second chain-homotopy of the form
\begin{equation}\label{eq:chain-Hmatrix}
\begin{bmatrix}
     \bH_{1,-2} & 0 \\
      0 & \bH_{0,-3}
\end{bmatrix} : \bC_{1} \oplus \bC_{0} \to  \bC_{-2} \oplus \bC_{-3}
\end{equation}
where $\bH_{1,-2}$ and $\bH_{0,-3}$ are defined in the same way as
$\bH_{2,-1}$ above, then we find that  $\Phi_{-1}\comp\Psi$ is
chain-homotopic to
\[
\begin{bmatrix}
      \bT_{1,-2} & 0 \\
     \MATHBF{X} & \bT_{0,-3} 
\end{bmatrix} : \bC_{1} \oplus \bC_{0} \to  \bC_{-2} \oplus \bC_{-3}.
\]
This lower-triangular map is a chain-homotopy equivalence because the
diagonal entries $\bT_{i,i-3}$ are.

We now introduce the quantum filtrations. For this, we suppose that
both 
$(K,N)$ and $(K,N')$ are pseudo-diagrams with small self-intersection
numbers (Definition~\ref{def:small-squares}). 
We wish to see whether these chain maps and chain homotopies
respect the quantum filtrations. For this, we need to strengthen our
hypotheses once more.

\begin{definition}
    \label{def:butter}
Let $N$ be a set of crossings and let $N'$ be a subset obtained by
forgetting one crossing. We say that the pair $(N, N')$ is \emph{admissible}
if the following holds. First, we require
that both $(K,N)$ and $(K,N')$ are pseudo-diagrams. In addition, we 
require that
whenever $(v,u)$ is a pair of points in $\Z^{N}$ of type at most $3$
with  $v(N'), u(N') \subset \{0,1\}$, we have $S_{vu} \le 6$.
\end{definition}

When $(N,N')$ is admissible in this sense, then both $(K,N)$ and
$(K,N')$ have small self-intersection numbers. The definition 
is set up so that it applies to all the pairs $(v,u)$
corresponding to the maps $f_{vu}$ which are involved in the
chain-maps $\Psi$ and $\Phi_{i}$ defined above, for these chain maps have
matrix entries $\bF_{ij}$ with $i-j=0,1$ or $2$, so they ultimately
rest on cobordisms $S_{vu}$ with $(v,u)$ of type at most
$2$. Furthermore, the chain-homotopies such as \eqref{eq:chain-H22},
\eqref{eq:chain-L} and \eqref{eq:chain-Hmatrix} involve only
cobordisms $S_{vu}$ of type at most $3$, so the condition in the
definition 
covers them also.  The 
$q$-grading, defined using the crossing-set $N$, is defined on all
$C_{v}$ for which $K_{v}$ is an unlink. In particular, $q$ is defined
on each $\bC_{i}$. 

\begin{proposition}\label{prop:ad-cubes}
    If $(N,N')$ is admissible in the sense of
    Definition~\ref{def:butter}, then the chain map
     \[
             \Psi : \bC_{1}\oplus \bC_{0} \to \bC_{-1}
      \]
     has order $\le 0$ with respect to the quantum filtration and is
     an isomorphism in the category $\hFD$.
\end{proposition}

\begin{proof}
    We note first that the maps $\bT_{2,-1} : \bC_{2} \to\bC_{-1}$ etc.~given
    by the cylindrical cobordisms are isomorphisms in $\hFD$. This is a
    corollary of
    Proposition~\ref{prop:isotopy-cube} and the 3-periodicity of the $q$-grading described in
   Lemma~\ref{lem:q-is-periodic}.
  The question of whether $\Psi$,
    $\Phi_{2}$ and $\Phi_{-1}$ preserve the filtrations is just the question of
    whether the maps $\bF_{ij}$ that appear as components of the maps
    \eqref{eq:Psi-def} and \eqref{eq:Phi-def} respect the filtration
    defined by $q$. The proof is exactly the same as the proof of
    Proposition~\ref{prop:C-is-filtered}. Similarly, the chain
    homotopies  \eqref{eq:chain-H22},
\eqref{eq:chain-L} and \eqref{eq:chain-Hmatrix} preserve the
filtration. So in the category $\hFD$, we have
\[
\begin{aligned}
    \Psi \circ \Phi_{2} &\simeq \bT_{2,-1} \\
    \Phi_{-1} \circ \Psi &\simeq \begin{bmatrix}
        \bT_{1,-2} & 0 \\
        \MATHBF{X} & \bT_{0,-3} \end{bmatrix}
\end{aligned}
\]
and the maps on the right are isomorphisms in $\hFD$.
\end{proof}

Just as the quantum grading $q$ is defined on $\bC(N)$, so we define a
quantum grading $q'$ on $\bC(N')$. 
The complexes $\bC(N')$ and $\bC_{2}$ are chain-homotopy
equivalent, as
noted above, but the gradings $q'$ and $q$ may apparently differ,
because the correction terms involved in the definition depend on the
original choice of a set of crossings. In fact, however, the gradings
do coincide:

\begin{lemma}\label{lem:same-grading-q}
    With the quantum gradings $q$ and $q'$ corresponding to the
    crossing-sets $N$ and $N'$ respectively, the complexes $\bC_{2}$, $\bC_{-1}$
    and $\bC(N')$ are isomorphic in the category $\hFD$. The
    isomorphisms are the maps $\bT$ given by the cylindrical cobordisms.
\end{lemma}

\begin{proof}
    We have already noted the isomorphism between $\bC_{2}$ and
    $\bC_{-1}$ in the proof of the previous proposition. For $\bC(N')$,
    by an application of Proposition~\ref{prop:isotopy-cube}, we may
    reduce to the case that the auxiliary data for $(K, N')$ is chosen
    so that $\bC(N')$ and $\bC_{2}$ coincide as complexes and the map
    $\bT$ is the identity map. We must then compare the definitions of
    the quantum gradings.  Choose any $v':N'\to \{0,1\}$ and let $v$
    denote the function $v(c)=v'(c)$ for $c\in N'$ and $v(c^*)=2$. Let
    $\epsilon$ denote the sign of $c^*$.  For a generator in $\bC_{2}
    \cong \bC(N')$ the two $Q$-gradings agree so the difference of the
    $q$-grading is
\[
\begin{aligned}
    q-q' &=  - v(c^{*}) +\frac{3}{2}(\sigma(v,o)-\sigma(v',o'))
           - (n_{+} - n'_{+}) + 2 ( n_{-} - n'_{-})
 \\
    &= -2+\frac{3}{2}(\sigma(v,o)-\sigma(v',o'))+\frac{3}{2}(1-\epsilon)-1
 \end{aligned}
\]
 (cf. equation~\eqref{eq:q-def}).  If $c^*$ is a positive crossing then the oriented resolution of $K$
 has $o(c^*)=0$ while if $c^*$ is a negative crossing then the oriented resolution has $o(c^*)=1$.  Thus
 for a positive crossing we have $\sigma(v,o)-\sigma(v',o')=2$ and
 $1-\epsilon=0$ while for a negative crossing
 we have $\sigma(v,o)-\sigma(v',o')=0$ and $1-\epsilon=2$.  In either case above difference is zero.
\end{proof}

Putting together this lemma and the previous proposition, we have:

\begin{proposition}
    If $(N,N')$ is admissible in the sense of
    Definition~\ref{def:butter}, then $\bC(N)$ and $\bC(N')$ are
    isomorphic in $\hFD$. \qed
\end{proposition}

\section{Dropping two crossings}

Consider again a collection of crossings $N$ for $K\subset \R^{3}$,
and let $N'$ and $N''$ be obtained from $N$ by dropping first one and
then a second crossing:
\[
\begin{aligned}
       N &=  N'' \cup \{c_{1}, c_{2}\} \\  
       N' &=  N'' \cup \{c_{2}\} 
\end{aligned}
\]

\begin{lemma}
    Suppose that all the links $K_{v}$ corresponding to vertices of
    $\bC(N)$, $\bC(N')$ and $\bC(N'')$ are unlinks. Suppose also that
    for all pairs $(v,u)$ in $\{0,1\}^{N}$ the corresponding cobordism
    $S_{vu}$ is orientable. Then the pairs $(N, N')$ and $(N', N'')$
    are both admissible in the sense of Definition~\ref{def:butter}.
\end{lemma}

\begin{remark}
    The orientability hypothesis in the lemma is equivalent to asking
    that the cobordism from $\{1\}^{N}$ to $\{0\}^{N}$ is orientable,
    since all the others are contained in this one.
\end{remark}

\begin{proof}[Proof of the lemma]
   Let $o \in \{0,1\}^{N}$ be the oriented resolution (or indeed any
   chosen point in this cube). Consider  $
   \sigma(o,v)$ as a function of $v$.  It is well-defined on all $v$ with $v(N'') \subset
   \{0,1\}$, because the corresponding links $K_{v}$ are unlinks. 
   For $v\in \{0,1\}^{N}$ we have $\sigma(o,v)=0$, because $S_{ov}$ is
   orientable.

    If $v'\in \Z^{N}$ has $v'(N')\subset \{0,1\}$ while $v'(c_{1})=-1$, then
   $\sigma(o,v')=0$ or $2$. To see this, consider because the unique $v\in
   \{0,1\}^{N}$ with $v|_{N'}=v'_{N'}$ $v(c_{1})=0$. By additivity, we
   have $\sigma(o,v') =
   \sigma(v,v')$. On the other hand, $\sigma(v,v') = 0$ or $2$, according as the cobordism $S_{vv'}$
   from the unlink $K_{v}$ to the unlink $K_{v'}$ is orientable or
   not. 

   Similarly, if $v''\in\Z^{N}$ has $v''(N'') \subset
   \{0,1\}$ while $v''(c_{1})=v''(c_{2})=-1$, then $\sigma(o,v'')$ is
   either $0$, $2$ or $4$, as we see by comparing $\sigma(o,v'')$ to
   $\sigma(o,v')$ where $v'$ is an adjacent element with
   $v'(c_{2})=0$.

   With these observations in place, we can verify that $(N,N')$ and
   $(N',N'')$ are both admissible. For example, to show that $(N',
   N'')$ is admissible, consider $v' \ge u'$ in $\Z^{N'}$ of type at
   most $3$. We may regard $v'$ and $u'$ as elements of $\Z^{N}$ by
   extending them with $v'(c_{1}) = u'(c_{1})=-1$. Consider (as an
   illustrative case) the situation in which $v'(c_{2})=2$ and
   $u'(c_{2})=-1$. Let $\tilde v'$ be the element with $\tilde
   v'|_{N'}=v'|_{N'}$ and $\tilde v'(c_{2})=-1$. By
   Lemma~\ref{lem:3-step-sig}, we have $\sigma(v',\tilde v')=2$, while
   the observations in the previous paragraph give $\sigma(\tilde v',
   v')\le 4$. So $\sigma(v', u') \le 6$, as required for
   admissibility. 
\end{proof}

\begin{corollary}\label{cor:plane-diagram-criterion}
    Let $N$ be the set of all crossings for a link $K$ arising from a
    given plane diagram $D(K)$. Let $N'$ and $N''$ be obtained by
    dropping first a crossing $c_{1}$ and then a pair of crossings
    $\{c_{1}, c_{2}\}$. Suppose that $c_{1}$ and $c_{2}$ are adjacent
    crossings along a
    parametrization of some component of $K$, and are either both
    over-crossings or under-crossings along this component. 
    Then the pairs $(N,N')$ and
    $(N',N'')$ are admissible.
\end{corollary}

\begin{proof}
    If we take the $0$- or $1$- resolution at every crossing in
    $D(K)$, we obtain in every case an unlink, and the cobordisms
    between them are orientable. If we resolve all crossings except
    $c_{1}$, then we obtain a diagram of a link with only one
    crossing, so this is again an unlink. If we resolve all crossings
    except $c_{1}$ and $c_{2}$ then we obtain a 2-crossing
    diagram. This additional hypothesis in the statement of the
    Corollary ensures that this 2-crossing diagram is not
    alternating.  This guarantees that the diagram is still an
    unlink. Thus all the conditions of the previous lemma are satisfied.
\end{proof}

\begin{corollary}\label{cor:plane-diagram-iso}
    As in the previous corollary, let $N$ be the set of all crossings
    arising from a plane diagram $D(K)$, let $\{c_{1}, c_{2}\}$ be a
    pair of these crossings that are adjacent and are either
    both over-crossings or both under-crossings. Again let $N'=N
    \setminus \{c_{1}\}$ and $N'' = N \setminus \{c_{1},c_{1}\}$. Then
    there are quantum-filtration-preserving chain maps
    \[
                \bC(N) \stackrel{\Psi}{\longrightarrow} 
                     \bC(N') \stackrel{\Psi'}{\longrightarrow}
                     \bC(N'') 
    \]
     and
    \[
                \bC(N'') \stackrel{\Phi'}{\longrightarrow} 
                     \bC(N') \stackrel{\Phi}{\longrightarrow}
                     \bC(N) ,
    \]
    such that $\Phi$ is the inverse of $\Psi$ and $\Phi'$ is the
    inverse of $\Psi'$  in the category $\hFD$. In particular,
    $\bC(N)$, $\bC(N')$ and $\bC(N'')$ are isomorphic in $\hFD$. \qed
\end{corollary}

\section{Reidemeister moves}
\label{sec:reidemeister}

We now specialize to the crossing-sets that arise from a plane
diagram 
of our knot or link $K$. We regard a diagram as being an image of a
generic projection of $K$ into $\R^{2}$, with a labelling of each
crossing as ``under'' or ``over''. 
As we have mentioned, the set of crossings $N$
coming from a diagram $D$ of $K$ fulfills the conditions of
Definition~\ref{def:pseudo-diagram}; that is, $(K,N)$ is a
pseudo-diagram. Furthermore, $(K,N)$ has small self-intersection
numbers in the sense of Definition~\ref{def:small-squares}. Indeed,
all the cobordisms $S_{vu}$ are orientable, and therefore have
$S_{vu}\cdot S_{vu}=0$. After a choice of good auxiliary data, we
therefore arrive at a filtered complex $\bC(K,N)$, an object of our
category $\hFD$.

It follows from Proposition~\ref{prop:isotopy-cube} that the
isomorphism class of $\bC(K,N)$ in $\hFD$ depends at most on the
isotopy class of the diagram $D$: isotopic diagrams, with
different choices of auxiliary data, will yield isomorphic objects.
Our task is to show that the isomorphism class of $\bC(K,N)$ is
independent of the diagram altogether, and depend only on the link
type of $K$. For this, we need to consider Reidemeister moves.

So let $D_{1}$ and $D_{2}$ be two plane diagrams for the same link type,
differing by a single Reidemeister move. That is, we suppose that
$D_{1}$ and $D_{2}$ coincide outside a standard disk in the plane, and
that inside the disk they have the standard form corresponding to a
Reidemeister move of type I, II or III. Let $K_{1}$ and $K_{2}$ be the
(isotopic) oriented links in $\R^{3}$ whose projections are $D_{1}$ and $D_{2}$,
and let $N_{1}$ and $N_{2}$ be their crossing-sets. Associated to
these links and crossing-sets, after choices of good auxiliary data, we have
cubes $\bC(K_{1}, N_{1})$ and $\bC(K_2, N_2)$, two objects of $\hFD$.

\begin{proposition}
    When the diagrams differ by a Reidemeister move as above, 
    the filtered cubes $\bC(K_{1}, N_{1})$ and $\bC(K_{2}, N_{2})$ are
    isomorphic in the category $\hFD$.
\end{proposition}

\begin{proof}
    This is a straightforward consequence of
    Corollary~\ref{cor:plane-diagram-iso}. Thus, for example, in the
    case that $D_{1}$ and $D_{2}$ differ by a Reidemeister move of
    type III as shown in Figure~\ref{fig:Reidemeister3}, 
    then we pass from $(K_{1}, N_{1})$ to $(K_{2}, N_{2})$
    in three steps, as follows.
    \begin{enumerate}
    \item First, drop the two over-crossings $c_{1}$ and $c_{2}$ from the
        crossing set $N_{1}$ to obtain a smaller set of crossings
        $N''$.
    \item Second, apply an isotopy to $K_{1}$ relative to $B(N'')$.
    \item Third, introduce two new crossings to $N''$ to obtain the
        crossing-set $N_{2}$.
    \end{enumerate}
    \begin{figure}
        \centering
        \includegraphics[height=2.75 in]{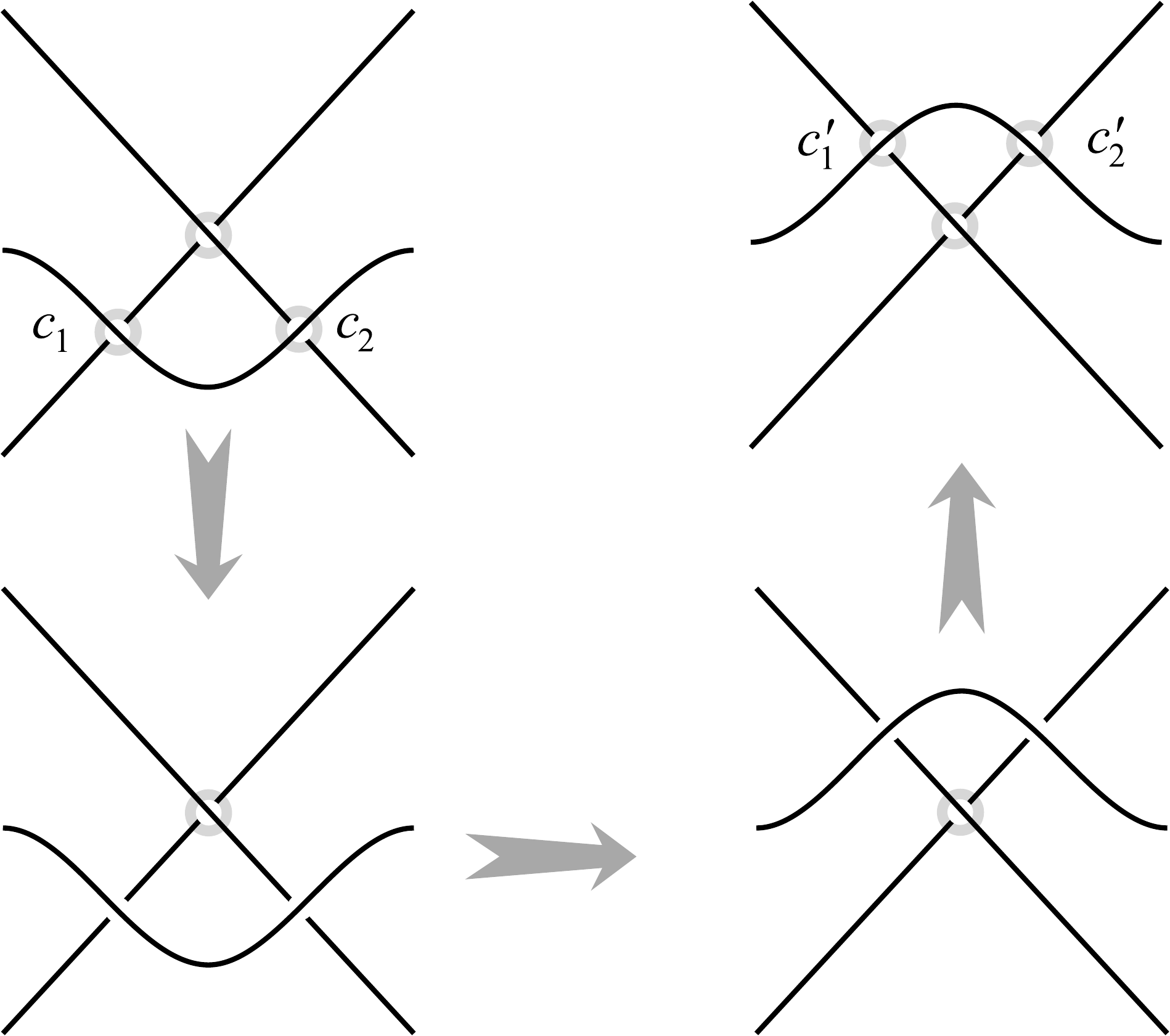}
        \caption{Dropping two crossings, isotoping a strand, and adding two crossings, to acheive Reidemeister III.}
        \label{fig:Reidemeister3}
    \end{figure}
    Corollary~\ref{cor:plane-diagram-iso} applies to the first and
    third steps, while Proposition~\ref{prop:isotopy-cube} applies
    to the middle step. The other types of Reidemeister moves are
    treated the same way.
\end{proof}

\begin{corollary}\label{cor:invariant-in-Cq}
    Let $K_{1}$ and $K_{2}$ be isotopic oriented links in $\R^{3}$ and
    let $N_{1}$ and $N_{2}$ be crossing-sets arising from diagrams
    $D_{1}$ and $D_{2}$ for these links. After choices of good
    auxiliary data, let $\bC(K_{i},N_{i})$ be the corresponding
    cubes ($i=1,2$). Then, with the filtrations defined by $q$, these
    cubes define isomorphic objects in the category
    $\hFD$. \qed
\end{corollary}

\section{The $h$-filtration on cubes}
\label{sec:h}

We now turn to the cohomological filtration $h$. Although there are a
few important differences, we can for the most adapt the sequence of
arguments that we have used for the $q$ filtration. This will lead us (for
example) to a version of Corollary~\ref{cor:invariant-in-Cq} for $h$.
 
To begin, we suppose that
$(K,N)$ is a pseudo-diagram (Definition~\ref{def:pseudo-diagram}).
On the cube $\bC(N)$, we define the $h$-grading by declaring
that the 
summand
$C_{v}$ has grading
\begin{equation}\label{eq:h-def}
h|_{C_{v}}= - \left( \sum_{c\in N} v(c) \right)+ \frac{1}{2} \sigma(v,o) +n_{-}.
\end{equation}
Here, as in equation~\eqref{eq:q-def} where we defined the $q$
grading, the term $n_{-}$ denotes the number of negative crossings and
$o$ denotes the oriented resolution.
As with $q$, we can extend the
definition of $h$ beyond the cube $\{0,1\}^{N}$ to all $v\in \Z^{N}$
for which $K_{v}$ is an unlink, because $\sigma(v,o)$ can be defined
for all such $v$.

Unlike $q$, the grading $h$ does not have period $3$ in each
coordinate. (See Lemma~\ref{lem:q-is-periodic}.) Instead we have the
following computation.

\begin{lemma}\label{lem:h-is-not}
    Suppose  $v$ is such that $K_{v}$ is an unlink, and suppose $v-u$ is divisible by
    $3$ in $\Z^{N}$ so that we have an isomorphism $\beta: C_{v} \to
    C_{u}$ 
     as abelian groups. Then
    the $h$-gradings on $C_{v}$ and $C_{u}$ are related by
    \[
       h|_{C_{v}} - h|_{C_{u}} = - \frac{ 2}{3} \sum_{c\in N} ( v(c) - u(c) )
     \]
\end{lemma}

\begin{proof}
    This is immediate from Lemma~\ref{lem:3-step-sig} and the formula
    for $h$.
\end{proof}

The $q$-grading has constant parity on $\bC$, but the $h$-grading does
not. The next lemma shows that $(\bC, \bF)$ can be regarded as a
$\Z/2$-graded complex with grading given by $h$.

\begin{lemma}\label{lem:h-is-odd}
    The differential $\bF$ on $\bC$ has odd degree with respect to the $\Z/2$
    grading defined by $h$ mod $2$.
\end{lemma}

\begin{proof}
    For $v\ge u$ in $\{0,1\}^{N}$, the corresponding component
    $f_{vu}$ of $\bF$ is obtained by counting instantons in
    $0$-dimensional moduli spaces $M_{vu}(\beta,\alpha)$ parametrized
    by a family of metrics $\bG_{vu}$  of dimension
    \[
    \begin{aligned}
        \dim \bG_{vu} &= -\chi(S_{vu}) - 1\\
        &= \sum_{c} (v(c) - u(c)) - 1.
    \end{aligned}   \]
     The links $K_{v}$ and $K_{u}$ are unlinks and the perturbations
     are chosen so that the critical points all have the same index
     mod $2$. So the fiber dimension of $M_{vu}(\beta,\alpha) \to \bG_{vu}$
     is independent of $\beta$ and $\alpha$ mod $2$, and is given by
     \eqref{eq:M-dim-bar}. Taking account of the dimension of the
     $\bG_{vu}$, we can therefore write
    \[
    \begin{aligned}
        \dim M_{vu} (\beta, \alpha) &= 8\kappa + \chi(\bar S_{vu}) +
        \frac{1}{2}( \bar
        S_{vu}\cdot\bar S_{vu}) + \dim \bG_{vu} \\
        &= 8\kappa + \chi(\bar S_{vu}) + \frac{1}{2}( \bar S_{vu}\cdot\bar S_{vu}) -
        \chi(S_{vu}) -1
    \end{aligned}
    \]
   mod $2$. For any closed surface $\bar S$ in $\R^{4}$, we have $\chi(\bar S) =
   (1/2)(\bar S \cdot \bar S)$ mod $2$. We also have
   $\kappa = \tfrac{1}{16}(S\cdot
     S)$ modulo $\tfrac{1}{2}\Z$. So the above formula can be written
     \[
          \dim M_{vu}(\beta,\alpha) = \frac{1}{2}(S_{vu} \cdot S_{vu}) - \chi(S_{vu})
          - 1
     \]
     mod $2$. So the component $f_{vu}$ can be non-zero only when
     $\frac{1}{2}(S_{vu} \cdot S_{vu}) + \chi(S_{vu})$ is odd. On the other hand,
     $\frac{1}{2}(S_{vu} \cdot S_{vu}) + \chi(S_{vu})$ is precisely the difference in
     $h$ mod $2$, between $C_{v}$ and $C_{u}$.
\end{proof}

As well as having odd degree, the differential $\bF$ preserves the
decreasing filtration defined by $h$:

\begin{proposition}\label{prop:C-is-filtered-by-h}
    If $(K,N)$ is a pseudo-diagram, then the differential $\bF
    : \bC \to \bC$ has order $\ge 1$ with respect to 
     the decreasing filtration defined by $h$.
\end{proposition}

\begin{proof}
    This follows from Lemma~\ref{lem:max-is-2}, as did the
    corresponding statement for the $q$-filtration,
    Proposition~\ref{prop:C-is-filtered}. Indeed, for any $v\ge u$ in
    the cube $\{0,1\}^{N}$, the corresponding map $f_{vu} : C_{v} \to
    C_{u}$ of $\bF$ satisfies
    \[
    \begin{aligned}
        \ord_{h} f_{vu} &\ge h|_{C_{u}} - h|_{C_{v}} \\ &= \left(
            \sum_{c\in N} ( v(c) - u(c)) \right)- \frac{1}{2}
        \sigma(v,u).
    \end{aligned}
     \]
      So we obtain $\ord_{h} f_{vu} \ge 0$ directly from
      Lemma~\ref{lem:max-is-2}. To actually obtain $\ord_{h} f_{vu}\ge
      1$, which is the desired result, we appeal to the fact that
      $\bF$ has odd degree with respect to the mod-2 grading, as we
      have seen in the previous lemma.
\end{proof}

\begin{remark}
    In contrast to the corresponding result for the $q$-filtration
    (Proposition~\ref{prop:C-is-filtered}), the proposition above does
    not require the hypothesis that $(K,N)$ has small
    self-intersection numbers.
\end{remark}

Let us now introduce the category $\mathfrak{C}_{h}$,  whose objects are
filtered, finitely-generated abelian groups with a
differential of order $\ge 1$, and whose morphisms are
chain maps of order $\ge 0$ modulo 
chain-homotopies of order $\ge -1$. When equipped with the filtration
defined by $h$, the complex $\bC=\bC(N)$, with its differential $\bF$, defines an object in
the category $\mathfrak{C}_{h}$, by the proposition above, though this
object is dependent on the set of crossings
$N$, and the choice of good auxiliary data. As in the case of
the $q$-filtration, we now wish to show that different choices lead to
isomorphic objects in this category; and as before, the essential step
is to compare $\bC(N)$ to $\bC(N')$, where $N'$
is obtained from $N$ by ``forgetting'' a crossing.

So we suppose again that $N' = N \setminus c^{*}$, and that both $(K,N)$
and $(K,N')$ are pseudo-diagrams. (We do not need to assume that these
pseudo-diagrams have small self-intersection numbers.)
The cubes $\bC(N)$
and $\bC(N')$ have gradings which we call $h$ and $h'$
respectively. On the other hand, we can identify $\bC(N')$ as before
with $\bC_{2}$, where $\bC_{i}$ for $i\in \Z$ is defined by
\eqref{eq:bCi-def} using the crossing-set $N$. In this way, we can
regard the grading $h$ as being defined also on $\bC(N')$. The
following two lemmas are the counterpart of 
Lemma~\ref{lem:same-grading-q} for the $h$-grading.

\begin{lemma}\label{lem:diff-grading-h}
    With the cohomological gradings $h$ and $h'$ corresponding to the
    crossing-sets $N$ and $N'$,  the filtered complexes $\bC_{2}$ and 
    $\bC(N')[1]$ are isomorphic in $\mathfrak{C}_{h}$. Here the
    notation $\bC[n]$ denotes the filtered complex obtained from $\bC$
    by shifting the filtration down by $n$ so that the map $\bC[n]\to
    \bC$ has order $\ge n$.
\end{lemma}

\begin{proof}
The lemma can again be reduced to the case that $\bC(N')$ and
$\bC_{2}$ are the same complex. So we need to prove that $h'=h+1$.
As in the proof of Lemma~\ref{lem:same-grading-q}, we write $\epsilon$
for the sign of the crossing $c^{*}$, and we have
    \[
\begin{aligned}
    h-h' &=  - v(c^{*}) +\frac{1}{2}(\sigma(v,o)-\sigma(v',o'))+
             n_{-} - n'_{-}
 \\
    &= -2+\frac{1}{2}(\sigma(v,o)-\sigma(v',o'))+\frac{1}{2}(1-\epsilon).
 \end{aligned}
\]
Once again, we have $o(c^{*})=0$ or $1$ according as $\epsilon$ is $1$
or $-1$ respectively. Thus
 for a positive crossing we have $\sigma(v,o)-\sigma(v',o')=2$ and
 $1-\epsilon=0$, while for a negative crossing
 we have $\sigma(v,o)-\sigma(v',o')=0$ and $1-\epsilon=2$.  In either case the difference is $-1$.
\end{proof}

Since the grading $h$ is not $3$-periodic, we get a different
result when we compare $\bC(N')$ with $\bC_{-1}$ instead of
$\bC_{2}$.

\begin{lemma}\label{lem:diff-grading-h2}
    With the cohomological gradings $h$ and $h'$ corresponding to the
    crossing-sets $N$ and $N'$,  the filtered complexes $\bC_{-1}$ and 
    $\bC(N')[-1]$ are isomorphic in $\mathfrak{C}_{h}$. 
\end{lemma}

\begin{proof}
    This follows from the previous lemma and Lemma~\ref{lem:h-is-not}.
\end{proof}

We can now state our main result about the effect of dropping a
crossing, for the $h$-filtration.

\begin{proposition}
    Suppose that $N'=N\setminus\{c^{*}\}$ and that both $(K,N)$ and
    $(K,N')$ are 
    pseudo-diagrams. Then the complexes
    $\bC(N)$ and $\bC(N')$, equipped with the
    filtrations defined by $h$ and $h'$, are isomorphic in the
    category $\mathfrak{C}_{h}$.
\end{proposition}

\begin{proof}
  We consider again the maps
   \eqref{eq:Psi-Phi-1} and  \eqref{eq:Psi-Phi-2}:
    \[
    \begin{gathered}
        \Psi: \bC_{1} \oplus \bC_{0} \to \bC_{-1} \\
        \Phi_{2}: \bC_{2} \to \bC_{1} \oplus \bC_{0}.
    \end{gathered}
    \]
     With respect to the grading $h$ on $\bC_{i}$, $i\in\Z$, the maps
     $\Psi$ and $\Phi_{2}$ have order $\ge 0$, just as in the proof of
     Proposition~\ref{prop:C-is-filtered-by-h}; and once again,
     because these maps have odd degree with respect to $h$ mod $2$,
     it actually follows that $\Psi$ and $\Phi_{2}$ have order $\ge 1$. 
     Via the isomorphisms in the previous two lemmas, the maps $\Psi$
     and $\Phi_{2}$ therefore become maps
    \[
    \begin{gathered}
        \Psi': \bC_{1} \oplus \bC_{0} \to \bC(N') \\
        \Phi'_{2}: \bC(N') \to \bC_{1} \oplus \bC_{0}.
    \end{gathered}
    \]
     of order $\ge 0$. They therefore define morphisms in $\mathfrak{C}_{h}$.

     If we look at the composite $\Psi \comp
     \Phi_{2} : \bC_{2} \to \bC_{-1}$, we see from the formula \eqref{eq:first-chain-composite}
     and the arguments below it that this chain map is chain-homotopic
     to the isomorphism $\bT_{2,-1}$ given by the cylindrical cobordisms. The
     chain-homotopies are the maps
      \[
            \bF_{2,-1} : \bC_{2} \to \bC_{-1}
      \]
      and 
      \[
               \bH_{2,-1} : \bC_{2} \to \bC_{-1}.
      \]
       Once again, with respect to the filtrations defined by $h$,
       these maps have order $\ge 0$, and therefore order $\ge 1$
       because of the mod $2$ grading. Using the isomorphisms of the
       previous two lemmas again, we obtain chain-homotopies
      \[
            \bF'_{2,-1} : \bC(N') \to \bC(N')
      \]
      and 
      \[
               \bH'_{2,-1} :  \bC(N') \to \bC(N').
      \]
      of order $\ge -1$.
      So $\Psi' \comp \Phi_{2}'$ is
       the identity morphism in the category $\mathfrak{C}_{h}$. The
       argument for $\Phi_{-1}' \comp \Psi'$ is very similar.
\end{proof}

With the above result about forgetting a single crossing, we can now
continue just as in the case of the $q$-filtration, to prove
invariance under Reidemeister moves. So we arrive at the following
statement, exactly analogous to Corollary~\ref{cor:invariant-in-Cq} above:

\begin{corollary}\label{cor:invariant-in-Ch}
    Let $K_{1}$ and $K_{2}$ be isotopic oriented links in $\R^{3}$ and
    let $N_{1}$ and $N_{2}$ be crossing-sets arising from diagrams
    $D_{1}$ and $D_{2}$ for these links. After choices of good
    auxiliary data, let $\bC(K_{i},N_{i})$ be the corresponding
    cubes ($i=1,2$). Then, with the decreasing filtrations defined by $h$, these
    cubes define isomorphic objects in the category
    $\mathfrak{C}_{h}$. \qed
\end{corollary}

\section{Proof of the main theorems}
\label{sec:proofs}

The proofs of Theorem~\ref{thm:main}, Theorem~\ref{thm:invariance} and
Proposition~\ref{prop:order-of-map} 
are now just a matter of pulling together the above material with the
results of \cite{KM-unknot}, as we now explain.

Given a diagram $D$ for an oriented link $K$, we obtain from it a set
of crossings $N$. The pair $(K,N)$ is a pseudo-diagram, and after a
choice of perturbations we obtain a complex $\bC=\bC(K,N)$ which
carries a filtration by $\Z\times\Z$ arising from the gradings $h$ and
$q$. We have seen that the differential $\bF$ on $\bC$ has order $\ge (
1,0)$, meaning that it has order $\ge 1$
with respect to $h$ and $\ge 0$ with respect to $q$
(Propositions~\ref{prop:C-is-filtered-by-h} and
\ref{prop:C-is-filtered} respectively). 

The $h$ and $q$ gradings decompose $\bC$ as
\[
         \bC = \bigoplus_{i,j} \bC^{i,j}
\]
(where $j$ corresponds to the $q$ grading). Let us write 
$\bF^{0}$ for the sum of
those terms of  $\bF$ which preserve $q$: so for each $j$ we have
\[
 \bF^{0} :  \bigoplus_{i} \bC^{i,j} \to \bigoplus_{i} \bC^{i,j}.
\]

\begin{lemma}\label{lem:h-by-1}
    The map $\bF^{0}$ shifts the $h$ grading by exactly $1$, so that
\[
\bF^{0}(\bC^{i,j}) \subset \bC^{i+1,j}
\]
for all $i,j$. 
\end{lemma}

\begin{proof}
    We refer to the equations \eqref{eq:q-f-order} and \eqref{eq:h-def} . Because our
    collection of crossings comes from a diagram, all
    self-intersection numbers are zero, and the formula can therefore 
    be rewritten
    \begin{equation}\label{eq:q-f-order-rewrite}
        \ord_{q} f_{vu} \ge 
       h|_{C_{u}} - h|_{C_{v}} - 1.
    \end{equation}
    So for a non-zero map $f_{vu}$ that contributes to $\bF^{0}$,
    we have
    \[
    0 \ge h|_{C_{u}} - h|_{C_{v}} - 1,
    \]
    which implies that the difference in the $h$-gradings is exactly $1$, because we always
    have the reverse inequality.
\end{proof}

As
the leading term of $\bF$ with respect to the $q$-filtration, $\bF^{0}$ has
square zero. So $(\bC, \bF^{0})$ is a bigraded complex whose
differential has bidegree $(1,0)$. 

\begin{proposition}
    With its bigrading supplied by $h$ and $q$, the bigraded complex $(\bC,
    \bF^{0})$ is isomorphic to  Khovanov's complex $(\bC(D^{\dag}),
    d_{\kh})$ associated to the diagram $D^{\dag}$ for $K^{\dag}$
    (obtained from the diagram $D$ for $K$ by changing all
    over-crossings to under-crossings and vice versa).
\end{proposition}

\begin{proof}
    Recall that the complex $\bC=\bC(N)$ has summands $C_{v}$ indexed
    by $v\in\{0,1\}^{N}$, and that $C_{v}$ is the chain complex that
  computes $\Isharp$ for the unlink $K_{v}$. We have chosen
  perturbations so that the differential on $C_{v}$ is zero. So if
  $K_{v}$ has $p(v)$ components, then (as in
  section~\ref{subsec:unlinks}) we have
\[
\begin{aligned}
    C_{v} &= \Isharp(K_{v}) \\
    &= H_{*}(S^{2})^{\otimes p(v)} 
\end{aligned}
\]
by the isomorphism \eqref{eq:rep-homology}.
The isomorphism depends on a choice of meridians, as well as a choice of a
standard Morse function on the product of $S^{2}$'s and a choice of
perturbations for the instanton equations. Via the
identification
\[
             s: H_{*}(S^{2}) \to V
\]
 of $H_{*}(S^{2})$ with the rank-2 group $V=\langle
\vp,\vm\rangle$, this becomes an isomorphism
\[
        s^{\otimes p} \comp \beta:  C_{v} \to V^{\otimes p(v)}
\]
Allowing for the
change in orientation convention, this gives an isomorphism of groups
\begin{equation}\label{eq:cube-iso-1}
         \boldsymbol{\beta}:  \bC \to \bC(D^{\dag})
\end{equation}
 We also see that the
   formulae for the $h$-grading and $q$-grading on $\bC$, given by
   \eqref{eq:h-def} and \eqref{eq:q-def} respectively, coincide with Khovanov's
   cohomological and quantum grading as defined in \cite{Khovanov}. (The terms
   $\sigma(v,o)$ in those formulae are absent in the case that the
   cube arises from a diagram, because all the self-intersection
   numbers are zero.)

    On the other hand, it 
     is shown in \cite{KM-unknot} that the spectral sequence
    associated to the filtered complex $(\bC, \bF)$, with its
    filtration by $h$, has $E_{1}$ page isomorphic to Khovanov's complex:
    that is, there is an isomorphism
\[
           (E_{1}, d_{1}) \stackrel{\cong}{\to} (\bC(D^{\dag}), d_{\kh})          
\]
   as groups with differential. With our choice of perturbations, 
    the elements in a given $h$-grading all
   have the same degree mod $2$, so the $d_{0}$ differential is absent
   and we have $E_{1} = \bC$. Furthermore, the $d_{1}$ differential is
   precisely the part of $\bF$ that increases $h$-grading by exactly
   $1$. Let us call this $\bF_{e}$. So we have
\begin{equation}\label{eq:cube-iso-2}
  \boldsymbol{\gamma}: (\bC, \bF_{e})  \stackrel{\cong}{\to} (\bC(D^{\dag}), d_{\kh}) .
\end{equation} 
  The isomorphism $\boldsymbol{\gamma}:\bC \to \bC(D^{\dag})$ in \eqref{eq:cube-iso-2} is
  not the same isomorphism as the map $\boldsymbol{\beta}$ in 
   \eqref{eq:cube-iso-1}. Rather, on each
  component $C_{v}$, it is defined by the map $\gamma$ from
  \eqref{eq:unknot-V-tensor-p}. 
   As explained in section~\ref{subsec:unlinks}, the isomorphism
   $\gamma$ and $\beta$ may differ. Both $\boldsymbol{\beta}$ and
   $\boldsymbol{\gamma}$ respect the homological grading $h$, but only
   $\boldsymbol{\beta}$ respects the $q$-grading.  The two
  isomorphisms differ by terms that strictly increase $q$, for on each
  $C_{v}$ we have
  \[
             \gamma - s^{\otimes p} \comp \beta = (\epsilon_{1} +
             \cdots) \comp \beta
   \]
   where $\epsilon_{1}$ etc.~are as in \eqref{eq:epsilon-terms}.
   In particular then, the map $\boldsymbol{\gamma}$ gives rise to a
   map on the associated graded objects with respect to the $q$
   filtration, which we write as 
   \begin{equation}\label{eq:cube-iso-3}
     \gr(\boldsymbol{\gamma}): \gr (\bC, \bF_{e})  \to (\bC(D^{\dag}), d_{\kh}) .
    \end{equation} 
   On the left, we can identify the bigraded complex $\gr (\bC,
   \bF_{e})$ with $(\bC, \bF')$, where $\bF'$ is obtained from
   $\bF_{e}$ by keeping only those summands that preserve $q$. This
   differential $\bF'$ coincides with $\bF^{0}$ by
   Lemma~\ref{lem:h-by-1}. So the map \eqref{eq:cube-iso-3} can be
   interpreted as a
   map
   \[
       \gr(\boldsymbol{\gamma}): (\bC, \bF^{0})  \to (\bC(D^{\dag}), d_{\kh}) 
    \]
   which respects the bigrading. The map $\bC \to \bC(D^{\dag})$ which
   appears here is the associated graded map arising from
   $\boldsymbol{\gamma}$, and is therefore equal to
   $\boldsymbol{\beta}$. 
Thus $\boldsymbol{\beta}$ in \eqref{eq:cube-iso-1} intertwines the
   differential $\bF^{0}$ with $d_{\kh}$.
\end{proof}

Theorem~\ref{thm:main} is a rewording of the above proposition: we
identify $\bC$ with $\bC(D^{\dag})$ by the isomorphism
\eqref{eq:cube-iso-1} and write the differential $\bF$ as
$d_{\sharp}$.

\begin{proof}[Proof of Theorem~\ref{thm:invariance}]
The theorem asserts that  $(\bC, \bF)$ is independent
of the choices made, up to isomorphism in the category
$\mathfrak{C}$. The choices here are again a diagram for $K$ with its
associated collection of crossings, and a choice of good auxiliary
data. 
Given two sets of choices and two corresponding objects
$(\bC_{1} , \bF_{1})$ and $(\bC_{2} , \bF_{2})$ in the category
$\mathfrak{C}$, we have seen already that these objects are isomorphic
in both the category $\mathfrak{C}_{h}$  and  $\mathfrak{C}_{q}$
(Corollaries \ref{cor:invariant-in-Ch} and \ref{cor:invariant-in-Cq}
respectively). An examination of the proofs shows that these objects
are isomorphic also in $\mathfrak{C}$. Indeed, the proofs were
obtained by exhibiting chain maps
\[
\begin{aligned}
 \Phi: (\bC_{1} , \bF_{1}) &\to (\bC_{2} , \bF_{2}) \\
  \Psi: (\bC_{2} , \bF_{2})& \to (\bC_{1} , \bF_{1})
\end{aligned}
\]
and chain-homotopy formulae
\[
           \begin{aligned}
                          \Phi \comp \Psi - \MATHBF{1} & = \bF \comp \Pi + \Pi \comp
           \bF  \\
                          \Psi \comp \Phi - \MATHBF{1} & = \bF \comp \Pi' + \Pi' \comp
           \bF  .
           \end{aligned}
\]
The point here is that the \emph{same} chain-maps and chain-homotopies
are used in the proofs of both Corollary~\ref{cor:invariant-in-Ch} and
Corollary~\ref{cor:invariant-in-Cq}. Thus from the proof of
Corollary~\ref{cor:invariant-in-Ch} we learn that with respect to the
$h$-filtration, the chain-maps have order $\ge 0$ and the
chain-homotopies have order $\ge -1$; while from the proof of
Corollary~\ref{cor:invariant-in-Cq} we learn that, with respect to the
$q$-filtration, the chain-maps have order $\ge 0$ and the
chain-homotopies have order $\ge 0$. It follows that $\Phi$ and $\Psi$
are mutually-inverse isomorphisms in the category $\mathfrak{C}$,
where the maps are chain-maps of order $\ge (0,0)$ up to
chain-homotopies of order $\ge (-1,0)$. This proves
Theorem~\ref{thm:invariance}. 
\end{proof}

\begin{proof}[Proof of Proposition~\ref{prop:order-of-map}]

Because the maps $\Isharp(S)$ can be computed as composite maps
when the cobordism $S$ is a composite, it is sufficient to treat
the cases that $S$ corresponds to the addition of a single handle,
of index $0$, $1$, or $2$. The cases $0$ and $2$ correspond to the
addition or removal of a single extra unlinked component, and are
straightforward. So we consider the case of index $1$.

In the index-1 case, we can form a link $K_{2}$ with plane diagram
having a crossing-set $N$ in such a way that $K_{1}$ and $K_{0}$ are
obtained from $(K_{2}, N)$ by resolving a distinguished crossing
$c_{*}\in N$ in two different ways. The links $K_{2}$, $K_{1}$ and
$K_{0}$ are to be oriented independently and arbitrarily. 
The set $N' = N\setminus\{c_{*}\}$
is then the set of crossings for diagrams of both $K_{1}$ and
$K_{0}$. By a straightforward generalization of the commutativity of
the square \eqref{eq:square-commutes}, we know that the map
$\Isharp(S)$ arises from a chain map
\[
           \bF_{10} : \bC_{1} \to \bC_{0}.
\]
We can regard this map as one term in the differential $\bF(K_{2},N)$ 
on the cube $\bC(K_{2},N)=\bC_{1}
\oplus \bC_{0}$ corresponding to $(K_{2}, N)$. (See
\eqref{eq:F-matrix}.) With respect to the $h$- and $q$-filtrations
which arise from $(K_{2}, N)$, we know that $\bF_{10}$ has order $\ge
(1,0)$. Let us denote the corresponding $h$- and $q$-gradings on the
abelian group $\bC_{1} \oplus \bC_{0}$ by $h_{2}$ and $q_{2}$. On
$\bC_{1}$ and $\bC_{0}$ we also have gradings $h_{1}, q_{1}$ and
$h_{0}, q_{0}$ respectively, arising from $(K_{1}, N')$ and $(K_{0}, N')$.

We examine the quantum gradings. With respect to $q_{2}$, the order of
$\bF_{10}$ is $\le 0$. On a summand $C_{v} \subset \bC_{1}$,  we have
\[
\begin{aligned}
    q_{2} - q_{1} &= -v(c_{*}) - (n^{2}_{+} - n^{1}_{+}) + 2
    (n^{2}_{-} - n^{1}_{-})\\
    &= -1 - (n^{2}_{+} - n^{1}_{+}) + 2 (n^{2}_{-} - n^{1}_{-}),
\end{aligned}
\]
where $n^{2}_{+}$ denotes the number of crossings in $N$ that are
positive for the chosen orientation of $K_{2}$ and so on. Similarly,
on a summand $C_{u} \subset \bC_{0}$ we have
\[
\begin{aligned}
    q_{2} - q_{0} &= -u(c_{*}) - (n^{2}_{+} - n^{1}_{+}) + 2
    (n^{2}_{-} - n^{1}_{-})\\
    &= - (n^{2}_{+} - n^{1}_{+}) + 2 (n^{2}_{-} - n^{1}_{-}),
\end{aligned}.
\]
So with respect to the filtrations defined by $q_{1}$ and $q_{0}$, the
order of $\bF_{10}$ is greater than or equal to the difference of the
above two expressions, which is
\[
 -1 + (n^{1}_{+} - n^{0}_{+}) - 2 (n^{1}_{-} - n^{0}_{-}).
\]
Since $\chi(S)=-1$, the proposition's assertion about the quantum
gradings will be proved if we show that
\[
          \frac{3}{2}(S\cdot S) = (n^{1}_{+} - n^{0}_{+}) - 2 (n^{1}_{-} - n^{0}_{-}).
\]
Since the number of crossings is the same for $K_{1}$ and $K_{0}$,
this is equivalent to:
\[
          S\cdot S = w_{1} - w_{0},
\]
where $w_{i} = n^{i}_{+} - n^{i}_{-}$ is the writhe of the diagram
$(K_{i}, N')$. At this point we should recall that $S\cdot S$ is
defined with respect to  framings of the boundaries $K_{1}$
and $K_{0}$ which have linking numbers zero with $K_{1}$ and
$K_{0}$. With respect to the blackboard framings of $K_{1}$ and
$K_{0}$, the self-intersection number of $S$ is zero. The writhe
measures the difference between the blackboard framing and the framing
with linking-number zero. So the result follows. The calculation for
the $h$-filtration is similar.
\end{proof}

\section{Examples}

\paragraph{Simple examples of pseudo-diagrams.}

To illustrate how one can work with pseudo-diagrams, we consider the
two oriented psuedo-diagrams in Figure~\ref{fig:pseudo-unlinks}.  The
first represents a 2-component unlink; the second is a Hopf link. Each
has one crossing, and Figure~\ref{fig:pseudo-unlinks} shows the two
different resolutions ($v=1,0$) in each case.  
\begin{figure}
    \centering
    \includegraphics[height=3 in]{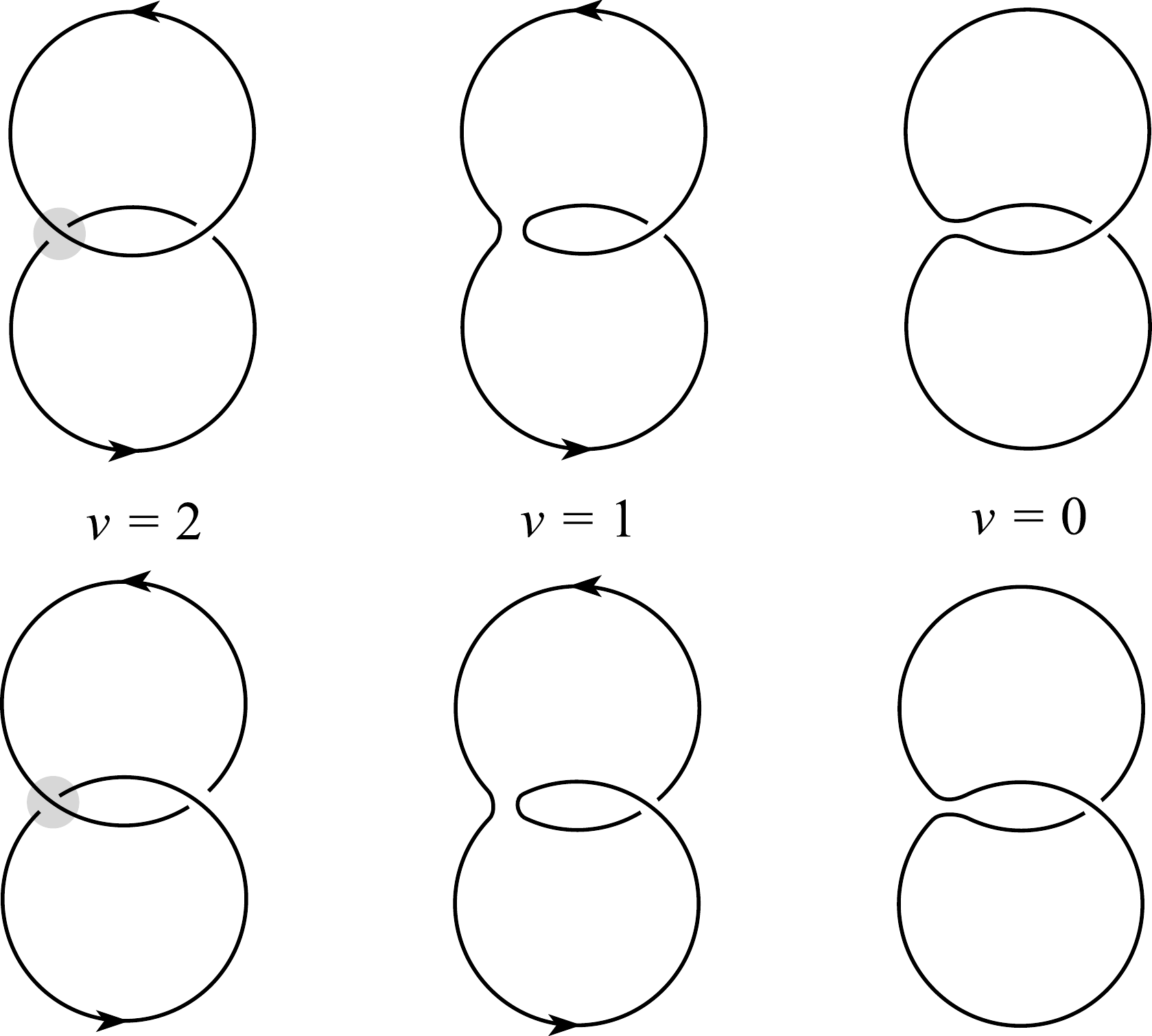}
    \caption{Pseudo-diagrams (left column) for an unlink and a Hopf
      link, each with one crossing. The resolution with the arrows is
      the oriented one in each case.}
    \label{fig:pseudo-unlinks}
\end{figure}

These examples show how the self-intersection number comes into play
when there are non-orientable cobordisms involved.  In the first
diagram we have $n_+=0$, $n_-=1$. (The total number of crossings is
$1$ in the pseudo-diagram.) The 1-resolution is the oriented
resolution.  The surface $S_{10}$ has $S_{10}\cdot S_{10}=2$, and hence
$\sigma(1,1)=0$ while $\sigma(0,1)=-2$.  Thus we see from the
definition of the $h$ and $q$ gradings on the cube $\bC$ that the
resolution with $v=1$ contributes two generators to the cube in
bi-gradings $(0,2)$ and $(0,0)$.  From $v=0$ there are also two
contributions, now in bi-gradings $(0,-2)$ and $(0,0)$.  Since the
$h$-gradings are all even, there can be no differential. The resulting
rank-4 group with its $h$- and $q$-gradings agrees with the Khovanov
homology of the unlink.

The Hopf link diagram still has $n_+=0$, $n_-=1$
but now $S_{10}\cdot S_{10}=-2$, and the $1$-resolution is still the
oriented resolution.  From $v=1$ there are two generators in
bi-gradings $(-1,-3)$ and $(-1,-1)$. From $v=0$ there are also two
generators, now with bi-gradings $(1,1)$ and $(1,3)$.  This
reproduces the Khovanov homology of the Hopf link (with the given
orientations.)

\paragraph{An example of a non-trivial differential, the $(4,5)$-torus knot.}

We will deal with the reduced case $\Inat(K)$ and work with rational
coefficients.  For the torus knot $T(4,5)$, the reduced Khovanov homology
is known and has rank 9. Below is a plot indicating with bullet-points
where the non-zero
groups are. These are plotted in the plane with coordinates $i$ and $j-i$, where
$i$ is the $h$-grading and $j$ is the $q$-grading. 

\begin{center}
\begin{sseq}[grid=go]{0...9}{11...16}
\ssmoveto{0}{11}
\ssdropbull
\ssmoveto{2}{13}
\ssdropbull
\ssmoveto{4}{13}
\ssdropbull
\ssmoveto{6}{13}
\ssdropbull
\ssname{b}
\ssmoveto{3}{14}
\ssdropbull
\ssname{a}
\ssmoveto{8}{15}
\ssdropbull
\ssname{d}
\ssmoveto{5}{16}
\ssdropbull
\ssname{c}
\ssmoveto{7}{16}
\ssdropbull
\ssmoveto{9}{16}
\ssdropbull
\end{sseq}
\end{center}

Starting from any diagram, the reduced Khovanov homology is a page of
both the $h$- and $q$-spectral sequences converging to $\Inat(K)$, and
we can ask where the differentials (if any) may be for $K=T(4,5)$.
The grading $j-i-1$ on the Khovanov homology reduces to the canonical
mod $4$ grading on $\Inat$. (See \cite[section 8.1]{KM-unknot}.)  From
the figure, we read off the Betti numbers for the mod $4$ grading on
the Khovanov homology as
\begin{equation}\label{eq:kh-betti}
             3, \quad 1, \quad 2, \quad 3
\end{equation}
in gradings $0$, $1$, $2$ and $3$ mod $4$ respectively.
The
differentials in the spectral sequence all have degree $-1$ mod $4$
with respect to this grading. In these coordinates the higher
differentials in the spectral sequence are also constrained by
\[
 \Delta j-\Delta i  \ge -1,
\]
where $\Delta i$ and $\Delta j$ denote the change in the $h$ and
$q$-gradings, as follows from \eqref{eq:q-f-order}. (The $S\cdot S$
term is absent in \eqref{eq:q-f-order} because we are dealing with a
diagram rather than a pseudo-diagram.)

The generators of the instanton complex that computes $\Inat(K)$ are
obtained from the representation variety
\[
     \Rep^{\natural}(K,\bi) := \Rep(K^{\natural}, \bi) / \mathit{SO}(3),
\]
which can also be regarded as the fiber of the map $\Rep(K, \bi)\to
S^{2}$ given by the holonomy around a chosen meridian.  In the case of
$T(4,5)$, one can show that this representation variety is a union of
an isolated non-degenerate point (corresponding to the reducible
representation) and four Morse-Bott circles (corresponding to
irreducibles).  Thus after a small perturbation, we see that the
instanton complex that computes $\Inat(K)$ has 9
generators. Since 9 is also the rank of the Khovanov homology, it
seems at first possible that the homology groups are isomorphic. But a
closer examination shows that the four Morse-Bott circles contribute
the same rank to each of the gradings mod $4$ in the instanton
complex. (To compute the relative Morse index of the $4$ circles, one
can study the representation varieties $\Rep^{\natural}(T(4,5),\alpha)$, defined
analogously to $\Rep^{\natural}(T(4,5),\bi)$ but with holonomy around the
meridian of the knot given by
\[
       \exp 2 \pi i
       \begin{pmatrix}
           -\alpha & 0 \\ 0 & \alpha
       \end{pmatrix}
\] 
for $\alpha\in [0,1/4]$. As $\alpha$ increases, circles of critical
points are emitted from or absorbed into the unique reducible. The
Morse indices can be read off from the sequence of these events.)  So
if the instanton homology had rank $9$, then the Betti numbers
\eqref{eq:kh-betti} for the Khovanov homology would differ from each
other by at most $1$ (corresponding to the contribution of the
isolated point in the representation variety). Since this is not the
case, if follows that the instanton homology has rank at most $7$.

On
the other hand, 
according to \cite[Proposition 1.4]{KM-unknot},  for a knot $K$ there is an
isomorphism between $\Inat(K)$ and the sutured instanton knot
Floer homology group $\mathit{KHI}(K)$, respecting the mod $4$
grading; and
from \cite[Corollary~1.2]{KM-alexander} we know that
the total rank of $\mathit{KHI}(K)$ is not less than the sum of the absolute values of the
coefficients of the Alexander polynomial, which for $T(4,5)$ is
\[
T^{6}-T^{5}+T^{2}-1+T^{-2}-T^{-5}+T^{-6}.
\]
This gives a lower bound of $7$ for the rank of $\Inat(K)$. We deduce
that in fact the rank is exactly $7$. So the spectral sequence has a
single non-zero differential which must have rank $1$.

The results of \cite{KM-alexander} give a little more information. The
coefficients of the Alexander polynomial arise as Euler characteristics of
the generalized eigenspaces of an operator $\mu$ on $\mathit{KHI}(K)$
which has degree $2$ with respect to the mod $4$ grading. It follows
that the Betti numbers in mod-$4$ gradings $l$ and $l+2$ are equal,
except for an offset term arising from the eigenvalue $0$. The
dimension of the generalized $0$-eigenspace is $1$, because it
corresponds to the coefficient of $T^{0}$. Inspecting
\eqref{eq:kh-betti} again, we see that the only possibility is that
the Betti numbers of the mod $4$ grading on $\Inat(K)$ are
\[
     2, \quad 1, \quad 2, \quad 2.
\]
The $0$-eigenspace contributes $1$ to the last of these. The
differential in the spectral sequence goes from $0$ to $3$ in the mod
$4$ grading. In terms of the diagram above, this means that the
differential must go from the row $j-i=13$ to the row $j-i=16$. The
authors do not know in which pages of the two spectral sequences the
differential appears. It is also not yet apparent what the ranks of the
associated graded groups are, for the $h$- and $q$-filtrations on $\Inat(K)$.

\bibliographystyle{abbrv}
\bibliography{q-grading}

\end{document}